\documentclass[11pt]{amsart}
\usepackage{amsmath, amssymb, amsthm, latexsym}
\usepackage{amsthm}
\usepackage{fullpage}
\usepackage{amssymb}
\usepackage{latexsym}
\usepackage[center]{caption}
\usepackage{tikz}
\usepackage{mathtools}
\usepackage{enumerate}
\usepackage{hyperref}

\newcounter{braid}
\newcounter{strands}
\pgfkeyssetvalue{/tikz/braid height}{1cm}
\pgfkeyssetvalue{/tikz/braid width}{1cm}
\pgfkeyssetvalue{/tikz/braid start}{(0,0)}
\pgfkeyssetvalue{/tikz/braid colour}{black}
\pgfkeys{/tikz/strands/.code={\setcounter{strands}{#1}}}

\makeatletter
\def\cross{%
  \@ifnextchar^{\message{Got sup}\cross@sup}{\cross@sub}}

\def\cross@sup^#1_#2{\render@cross{#2}{#1}}

\def\cross@sub_#1{\@ifnextchar^{\cross@@sub{#1}}{\render@cross{#1}{1}}}

\def\cross@@sub#1^#2{\render@cross{#1}{#2}}

\def\render@cross#1#2{
  \def\strand{#1}
  \def\crossing{#2}
  \pgfmathsetmacro{\cross@y}{-\value{braid}*\braid@h}
  \pgfmathtruncatemacro{\nextstrand}{#1+1}
  \foreach \thread in {1,...,\value{strands}}
  {
    \pgfmathsetmacro{\strand@x}{\thread * \braid@w}
    \ifnum\thread=\strand
    \pgfmathsetmacro{\over@x}{\strand * \braid@w + .5*(1 - \crossing) * \braid@w}
    \pgfmathsetmacro{\under@x}{\strand * \braid@w + .5*(1 + \crossing) * \braid@w}
    \draw[braid] \pgfkeysvalueof{/tikz/braid start} +(\under@x pt,\cross@y pt) to[out=-90,in=90] +(\over@x pt,\cross@y pt -\braid@h);
    \draw[braid] \pgfkeysvalueof{/tikz/braid start} +(\over@x pt,\cross@y pt) to[out=-90,in=90] +(\under@x pt,\cross@y pt -\braid@h);
    \else
    \ifnum\thread=\nextstrand
    \else
     \draw[braid] \pgfkeysvalueof{/tikz/braid start} ++(\strand@x pt,\cross@y pt) -- ++(0,-\braid@h);
    \fi
   \fi
  }
  \stepcounter{braid}
}

\tikzset{braid/.style={double=\pgfkeysvalueof{/tikz/braid colour},double distance=1pt,line width=2pt,white}}

\newcommand{\braid}[2][]{%
  \begingroup
  \pgfkeys{/tikz/strands=2}
  \tikzset{#1}
  \pgfkeysgetvalue{/tikz/braid width}{\braid@w}
  \pgfkeysgetvalue{/tikz/braid height}{\braid@h}
  \setcounter{braid}{0}
  \let\sigma=\cross
  #2
  \endgroup
}
\makeatother

\input xypic
\newtheorem{theorem}{Theorem}[section]

\newtheorem{lemma}[theorem]{Lemma}
\newtheorem{conjecture}[theorem]{Conjecture}
\newtheorem{corollary}[theorem]{Corollary}
\theoremstyle{definition}

\newtheorem{definition}[theorem]{Definition}

\newtheorem{remark}[theorem]{Remark}
\newtheorem{example}[theorem]{Example}


\def\Z{\mathbb{Z}}

\def\Pi{\mathbb{P}^{\infty}}

\def\Zpk{\mathbb{Z}/p^{k}}
\def\Zpk1{\mathbb{Z}/p^{k-1}}

\def\sl2{\widetilde{SL_{2}(\Z)}}

\DeclareMathOperator{\ord}{ord}

\DeclareMathOperator{\Gal}{Gal}

\DeclareMathOperator{\Frob}{Frob}

\DeclareMathOperator{\Res}{Res}

\DeclareMathOperator{\rank}{rank}
\DeclareMathOperator{\Sel}{Sel}
\DeclareMathOperator{\GL}{GL}
\DeclareFontFamily{U}{wncy}{}
\DeclareFontShape{U}{wncy}{m}{n}{<->wncyr10}{}
\DeclareSymbolFont{mcy}{U}{wncy}{m}{n}
\DeclareMathSymbol{\Sh}{\mathord}{mcy}{"58}

\usepackage[T2A,T1]{fontenc}
\DeclareSymbolFont{cyrillic}{T2A}{cmr}{m}{n}
\DeclareMathSymbol{\Sha}{\mathalpha}{cyrillic}{216}

\usepackage{enumitem}

\setlist[enumerate]{leftmargin=*}

\usepackage{array}

\title{Congruences between Heegner points and quadratic twists of elliptic curves}
\begin{document}
\author[Daniel Kriz]{Daniel Kriz}\email{dkriz@princeton.edu}
\address{Department of Mathematics, Princeton University, Fine Hall, Washington Rd, Princeton, NJ 08544}
\author[Chao Li]{Chao Li}\email{chaoli@math.columbia.edu} 
\address{Department of Mathematics, Columbia University, 2990 Broadway,
 New York, NY 10027}


\date{\today}

\maketitle

\begin{abstract}
 We establish a congruence formula between $p$-adic logarithms of Heegner points for two elliptic curves with the same mod $p$ Galois representation. As a first application, we use the congruence formula when $p=2$ to explicitly construct many quadratic twists of analytic rank zero (resp. one) for a wide class of elliptic curves $E$. We show that the number of twists of $E$ up to twisting discriminant $X$ of analytic rank zero (resp. one) is $\gg X/\log^{5/6}X$, improving the current best general bound towards Goldfeld's conjecture due to Ono--Skinner (resp. Perelli--Pomykala). We also prove the 2-part of the Birch and Swinnerton-Dyer conjecture for many rank zero and rank one twists of $E$, which was only recently established for specific CM elliptic curves $E$.
\end{abstract}


\section{Introduction}

\subsection{Goldfeld's conjecture} Let $E$ be an elliptic curve over $\mathbb{Q}$. We denote by $r_\mathrm{an}(E)$ its analytic rank. By the theorem of Gross--Zagier and Kolyvagin, the rank part of the Birch and Swinnerton-Dyer conjecture holds whenever $r_\mathrm{an}(E)\in\{0,1\}$. One can ask the following natural question: how is $r_\mathrm{an}(E)$ distributed when $E$ varies in families? The simplest (1-parameter) family is given by the quadratic twists family of a given curve $E$. For a fundamental discriminant $d$, we denote by $E^{(d)}$ the quadratic twist of $E$ by $\mathbb{Q}(\sqrt{d})$. The celebrated conjecture of Goldfeld \cite{Goldfeld1979} asserts that $r_\mathrm{an}(E^{(d)})$ tends to be as low as possible (compatible with the sign of the function equation). Namely in the quadratic twists family $\{E^{(d)}\}$, $r_\mathrm{an}$ should be 0 (resp. 1) for $50\%$ of $d$'s. Although $r_\mathrm{an}\ge2$ occurs infinitely often, its occurrence should be sparse and accounts for only $0\%$ of $d$'s. More precisely, 

\begin{conjecture}[Goldfeld]\label{conj:fullgoldfeld} Let $$N_r(E, X)=\{|d|<X: r_\mathrm{an}(E^{(d)})=r\}.$$ Then 
  for $r\in\{0,1\}$, $$N_r(E,X)\sim \frac{1}{2} \sum_{|d|<X}1,\quad X\rightarrow \infty.$$ Here $d$ runs over all fundamental discriminants.
\end{conjecture}

Goldfeld's conjecture is widely open: we do not yet know a single example $E$ for which Conjecture \ref{conj:fullgoldfeld} is valid. One can instead consider the following weaker version (replacing 50\% by any positive proportion):

\begin{conjecture}[Weak Goldfeld]\label{conj:weakgoldfeld}
  For $r\in \{0,1\}$, $N_r(E,X)\gg X$.
\end{conjecture}

\begin{remark}
  Heath-Brown (\cite[Thm. 4]{Heath-Brown2004}) proved Conjecture \ref{conj:weakgoldfeld}  \emph{conditional} on GRH. Recently, Smith \cite{Smith2017} has announced a proof (\emph{conditional} on BSD) of Conjecture \ref{conj:fullgoldfeld} for curves with full rational 2-torsion. In our recent work \cite{Kriz2016a}, we have proved Conjecture \ref{conj:weakgoldfeld} \emph{unconditionally} for any $E/\mathbb{Q}$ with a rational 3-isogeny.
\end{remark}


  When $r=0$, the best unconditional general result towards Goldfeld's conjecture is due to Ono--Skinner \cite{Ono1998}: they showed that for any elliptic curve $E/\mathbb{Q}$, $$N_0(E, X)\gg \frac{X}{\log X}.$$ When $E(\mathbb{Q})[2]=0$, Ono \cite{Ono2001} improved this result to $$N_0(E, X)\gg \frac{X}{\log^{1-\alpha} X}$$ for some $0<\alpha<1$ depending on $E$. When $r=1$, even less is known. The best general result is due to Perelli--Pomykala \cite{Perelli1997} using analytic methods: they showed that for any $\varepsilon>0$, $$N_1(E, X)\gg  X^{1-\varepsilon}.$$ Our main result improves both bounds, under a technical assumption on the 2-adic logarithm of the associated Heegner point on $E$.

Let us be more precise. Let $E/\mathbb{Q}$ be an elliptic curve of conductor $N$. Throughout this article, we will use $K=\mathbb{Q}(\sqrt{d_K})$ to denote an imaginary quadratic field of fundamental discriminant $d_K$ satisfying the \emph{Heegner hypothesis for $N$}:
\begin{center}
each prime factor $\ell$ of $N$ is split in $K$.
\end{center}
We denote by $P\in E(K)$ the corresponding Heegner point, defined up to sign and torsion with respect to a fixed modular parametrization $\pi_E: X_0(N)\rightarrow E$ (see \cite{Gross1984}). Let $$f(q)=\sum_{n=1}^\infty a_n(E) q^n\in S_2^\mathrm{new}(\Gamma_0(N))$$ be the normalized newform associated  to $E$. Let $\omega_E\in \Omega_{E/\mathbb{Q}}^1 := H^0(E/\mathbb{Q},\Omega^1)$ such that $$\pi_E^*(\omega_E)= f(q) \cdot dq/q.$$ We denote by $\log_{\omega_E}$ the formal logarithm associated to $\omega_E$. Notice $\omega_E$ may differ from the N\'{e}ron differential by a scalar when $E$ is not the optimal curve in its isogeny class.

Now we are ready to state our main result.

\begin{theorem}\label{thm:goldfeld}
  Suppose $E/\mathbb{Q}$ is an elliptic curve with $E(\mathbb{Q})[2]=0$. Suppose there exists an imaginary quadratic field $K$ be satisfying the Heegner hypothesis for $N$ such that
  \begin{equation}
    \label{eq:star}\
2\text{ splits in } K \text{ and }\quad \frac{|\tilde E^\mathrm{ns}(\mathbb{F}_2)|\cdot\log_{\omega_E}(P)}{2}\not\equiv0\pmod{2}.     \tag{$\bigstar$}
\end{equation}
Then for $r\in\{0,1\}$,  we have $$N_r(E, X)\gg
  \begin{cases}
   \displaystyle \frac{X}{\log^{5/6} X}, & \text{if }\Gal(\mathbb{Q}(E[2])/\mathbb{Q})\cong S_3, \\
   \displaystyle \frac{X}{\log^{2/3} X}, & \text{if }\Gal(\mathbb{Q}(E[2])/\mathbb{Q})\cong \mathbb{Z}/3 \mathbb{Z}.
  \end{cases}$$ 
\end{theorem}

\begin{remark}
  Assumption (\ref{eq:star}) imposes certain constraints on $E/\mathbb{Q}$ (e.g., its local Tamagawa numbers at odd primes are odd, see \S \ref{sec:strategy-proof}), but it is satisfied for a wide class of elliptic curves. See \S \ref{sec:examples} for examples and also Remark \ref{rem:assumptionstar} on the wide applicability of Theorem \ref{thm:goldfeld}.
\end{remark}

\begin{remark}\label{rem:MazurRubin}
  Mazur--Rubin \cite{Mazur2010} proved similar results for the number of twists of \emph{2-Selmer rank} $0,1$. Again we remark that it however does not have the same implication for analytic rank $r=0,1$ (or algebraic rank 1), since the $p$-converse to the theorem of Gross--Zagier and Kolyvagin for $p=2$ is not known.
\end{remark}

\begin{remark}
For certain elliptic curves with $E(\mathbb{Q})[2]=\mathbb{Z}/2 \mathbb{Z}$, the work of Coates--Y. Li--Tian--Zhai \cite{Coates2015} also improves the current bounds, using a generalization of the classical method of Heegner and Birch for prime twists.
\end{remark}

\subsection{Congruences between $p$-adic logarithms of Heegner points} The starting point of the proof of Theorem \ref{thm:goldfeld} is the simple observation that quadratic twists doesn't change the mod 2 Galois representations: $E[2]\cong E^{(d)}[2]$. More generally, suppose $p$ is a prime and $E, E'$ are  two elliptic curves with isomorphic semisimplified Galois representations $E[p^m]^{\mathrm{ss}}\cong E'[p^m]^{\mathrm{ss}}$ for some $m\ge1$, one expects that there should be a congruence mod $p^m$ between the special values (or derivatives) of the associated $L$-functions of $E$ and $E'$. It is usually rather subtle to formulate such congruence precisely. Instead, we work directly with the $p$-adic incarnation of the $L$-values -- the $p$-adic logarithm of Heegner points and we prove the following key congruence formula.


\begin{theorem}\label{thm:maincongruence} Let $E$ and $E'$ be two elliptic curves over $\mathbb{Q}$ of conductors $N$ and $N'$ respectively. Suppose $p$ is a prime such that there is an isomorphism of semisimplified  $G_\mathbb{Q}:=\Gal(\overline{\mathbb{Q}}/\mathbb{Q})$-representations $$E[p^m]^{\mathrm{ss}} \cong E'[p^m]^{\mathrm{ss}}$$ for some $m\ge1$. Let $K$ be an imaginary quadratic field satisfying the Heegner hypothesis for both $N$ and $N'$. Let $P \in E(K)$ and $P' \in E'(K)$ be the Heegner points. Assume $p$ is split in $K$. Then we have
$$\left(\prod_{\ell|pNN'/M}\frac{|\tilde{E}^{\mathrm{ns}}(\mathbb{F}_{\ell})|}{\ell}\right)\cdot \log_{\omega_E}P \equiv \pm\left(\prod_{\ell|pNN'/M}\frac{|\tilde{E}'^{,\mathrm{ns}}(\mathbb{F}_{\ell})|}{\ell}\right)\cdot\log_{\omega_{E'}}P' \pmod {p^m\mathcal{O}_{K_p}}.$$ Here 
$$M = \prod_{\ell|(N,N') \atop a_{\ell}(E)\equiv a_{\ell}(E')\pmod{p^m}}\ell^{\ord_{\ell}(NN')}.$$
\end{theorem}

\begin{remark}
  Recall that  $\tilde E^\mathrm{ns}(\mathbb{F}_\ell)$ denotes the number of $\mathbb{F}_\ell$-points of the nonsingular part of the mod $\ell$ reduction of $E$, which is $\ell+1-a_\ell(E)$ if $\ell\nmid N$, $\ell\pm1$ if $\ell|| N$ and $\ell$ if $\ell^2|N$. The factors in the above congruence can be understood as the result of removing the Euler factors of $L(E,1)$ and $L(E',1)$ at bad primes.
\end{remark}

\begin{remark}
The link between the $p$-adic logarithm of Heegner points and $p$-adic $L$-functions dates back to Rubin \cite{Rubin1992} in the CM case and was recently established in great generality by Bertolini--Darmon--Prasanna \cite{Bertolini2013} and Liu--S. Zhang--W. Zhang \cite{Liu2014}. However, our congruence formula is based on direct $p$-adic integration and does \emph{not} use this deep link with $p$-adic $L$-functions.   
\end{remark}

\begin{remark}
  Since there is no extra difficulty, we prove a slightly more general version (Theorem \ref{thm:genmaincongruence}) for Heegner points on abelian varieties of $\GL_2$-type. The same type of congruence should hold for modular forms of weight $k\ge2$ (in a future work), where the $p$-adic logarithm of Heegner points is replaced by the $p$-adic Abel--Jacobi image of generalized Heegner cycles defined in \cite{Bertolini2013}. 
\end{remark}

\subsection{A by-product: the 2-part of the BSD conjecture}

The BSD conjecture predicts the precise formula \begin{equation}
  \label{eq:bsdformula}
\frac{L^{(r)}(E/\mathbb{Q},1)}{r!\Omega(E/\mathbb{Q}) R(E/\mathbb{Q})}=\frac{\prod_p c_p(E/\mathbb{Q})\cdot |\Sha(E/\mathbb{Q})|}{|E(\mathbb{Q})_\mathrm{tor}|^2}  
\end{equation} for the leading coefficient of the Taylor expansion of $L(E/\mathbb{Q},s)$ at $s=1$ (here $r$ denotes the analytic rank) in terms of various important arithmetic invariants of $E$ (see \cite{Gross2011} for detailed definitions). The odd-part of the BSD conjecture has recently been established in great generality when $r\le1$, but very little (beyond numerical verification) is known concerning the \emph{2-part of the BSD conjecture} (BSD(2) for short). A notable exception is Tian's breakthrough \cite{Tian2014} on the congruent number problem, which establishes BSD(2) for many quadratic twists of $X_0(32)$ when $r\le1$. Coates outlined a program (\cite[p.35]{Coates2013}) generalizing Tian's method for establishing BSD(2) for many quadratic twists of a general elliptic curve when $r\le1$, which has succeeded for two more examples $X_0(49)$ (\cite{Coates2015}) and $X_0(36)$ (\cite{Cai2016}). We remark that all these three examples are CM with rational 2-torsion.

We now can state the following consequence on BSD(2) when $r\le1$ for many explicit twists, at least when the local Tamagawa number  at 2 is odd.

\begin{theorem}\label{thm:2partBSD}
  Let $E/\mathbb{Q}$ be an elliptic curve with $E(\mathbb{Q})[2]=0$. Assume there is an imaginary quadratic field $K$ satisfying the Heegner hypothesis for $N$ and Assumption (\ref{eq:star}). Further assume that the local Tamagawa number $c_2(E)$ is odd. If $E$ has additive reduction at 2, further assume its Manin constant is odd.

  Let $\mathcal{S}$ be the set of primes $$\mathcal{S}=\{\ell\nmid 2N: \ell \text{ splits in }K, \Frob_\ell\in \Gal(\mathbb{Q}(E[2])/\mathbb{Q})\text{ has order }3\}.$$
Let $\mathcal{N}$ be the set of all integers $d\equiv 1\pmod{4}$ such that $|d|$ is a square-free product of primes in $\mathcal{S}$. We have:
  \begin{enumerate}
  \item \label{item:2partK} If BSD(2) is true for $E/K$, then BSD(2) is true for $E^{(d)}/K$, for any $d\in\mathcal{N}$.
  \item \label{item:2partQ} If BSD(2) is true for $E/\mathbb{Q}$ and $E^{(d_K)}/\mathbb{Q}$, then BSD(2) is true for $E^{(d)}/\mathbb{Q}$ and $E^{(d\cdot d_K)}/\mathbb{Q}$, for any $d\in\mathcal{N}$ such that $\chi_d(-N)=1$.
  \end{enumerate}
\end{theorem}

\begin{remark}
  BSD(2) for a single elliptic curve (of small conductor) can be proved by numerical calculation when $r\le1$ (see \cite{Miller2011} for curves of conductor at most 5000). Theorem \ref{thm:2partBSD} then allows one to deduce BSD(2) for many of its quadratic twists (of arbitrarily large conductor). See \S \ref{sec:examples} for examples. 
\end{remark}

\begin{remark}
  Manin's conjecture asserts the Manin constant for any optimal curve is 1, which would imply that the Manin constant for $E$ is odd since $E$ is assumed to have no rational 2-torsion. Cremona has proved Manin's conjecture for all optimal curves of conductor at most $380000$ (see \cite[Theorem 2.6]{Agashe2006} and the update at \url{http://johncremona.github.io/ecdata/#optimality}). 
\end{remark}

\subsection{Structure of the paper}

  The main congruence (Theorem \ref{thm:maincongruence}) is proved in \S \ref{sec:proof-main-theorem}. We explain the ideal of the proof in \S \ref{sec:idea-argument}. In \S \ref{sec:goldf-conj-gener} we prove the application to Goldfeld's conjecture for general $E$ (Theorem \ref{thm:goldfeld}). In \S \ref{sec:proof-4} and \S \ref{sec:proof-5}, we prove the application to BSD(2) (Theorem \ref{thm:2partBSD}). In \S \ref{sec:examples}, we include numerical examples illustrating the wide applicability of Theorems \ref{thm:goldfeld} and \ref{thm:2partBSD}.

\subsection{Acknowledgments} We are grateful to J. Coates, B. Mazur and W. Zhang for their interest and helpful comments. Our debt to the two papers \cite{Bertolini2013} and \cite{Liu2014} should be clear to the readers. The examples in this article are computed using Sage (\cite{sage}).

\section{Proof of the main congruence}\label{sec:proof-main-theorem}

\subsection{The strategy of the proof}\label{sec:idea-argument}We first give the idea of the proof of Theorem \ref{thm:maincongruence}. From the congruent Galois representations, we deduce that the coefficients of the associated modular forms are congruent away from the bad primes in $pNN'/M$. After applying suitable stabilization operators (\S \ref{stabilization}) at primes in $NN'/M$, we obtain $p$-adic modular forms whose coefficients are all congruent. This congruence is preserved when applying a power $\theta^j$ of the Atkin--Serre operator $\theta$. Letting $j\rightarrow -1$ ($p$-adically) and using Coleman's theorem on $p$-adic integration (generalized in \cite{Liu2014}, see \S \ref{sec:coleman-integration}), we can identify the values of $\theta^{-1}f$ and $\log_{\omega_f}$ at CM points. The action of stabilization operators at CM points (\S \ref{sec:stabilizationCM}) gives rise to the extra Euler factors. Summing over the CM points finally proves the main congruence between $p$-adic logarithms of Heegner points (\S \ref{sec:proof-theorem}). This procedure is entirely parallel to the construction of anticyclotomic $p$-adic $L$-functions of \cite{Bertolini2013}, but we stress that the congruence itself (without linking to the $p$-adic $L$-function) is more direct and does \emph{not} require the main result of \cite{Bertolini2013}. In particular, we work on $X_0(N)$ directly (as opposed to working on the finite cover $X_1(N)$) and we do not require $E$ to have good reduction at $p$.

\subsection{$p$-adic modular forms}\label{sec:padicmodularforms} Henceforth, it will be useful to adopt Katz's viewpoint of $p$-adic modular forms as rules on the moduli space of isomorphism classes of ``ordinary test triples''. (For a detailed reference, see for example \cite[Chapter V]{Katz1976}.) 

\begin{definition}[Ordinary test triple]  Let $R$ be a $p$-adic ring (i.e. the natural map $R \rightarrow \varprojlim R/p^nR$ is an isomorphism). An \emph{ordinary test triple} $(A,C,\omega)$ over $R$ means the following:
\begin{enumerate}
\item $A/R$ is an elliptic curve which is ordinary (i.e. $A$ is ordinary over $R/pR$),
\item (level $N$ structure) $C \subset A[N]$ is a cyclic subgroup of order $N$ over $R$ such that the $p$-primary part $C[p^{\infty}]$ is the \emph{canonical subgroup} of that order (i.e., letting $\hat{A}$ be the formal group of $A$, we have $C[p^{\infty}] = \hat{A}[p^{\infty}]\cap C$),
\item $\omega \in \Omega_{A/R}^1 := H^0(A/R,\Omega^1)$ is a differential. 
\end{enumerate}
Given two ordinary test triples $(A,C,\omega)$ and $(A',C',\omega')$ over $R$, we say there is an \emph{isomorphism} $(A,C,\omega) \xrightarrow{\sim} (A',C',\omega')$ if there is an isomorphism $i : A \rightarrow A'$ of elliptic curves over $R$ such that $\phi(C) = C'$ and $i^*\omega' = \omega$. Henceforth, let $[(A,C,\omega)]$ denote the isomorphism class of the test triple $(A,C,\omega)$.
\end{definition}

\begin{definition}[Katz's interpretation of $p$-adic modular forms]\label{def:modforms} Let $S$ be a fixed $p$-adic ring. Suppose $F$ as a rule which, for every $p$-adic $S$-algebra $R$, assigns values in $R$ to \emph{isomorphism classes} of test triples $(A,C,\omega)$ of level $N$ defined over $R$. As such a rule assigning values to isomorphism classes of ordinary test triples, consider the following conditions:
\begin{enumerate}
\item (Compatibility under base change) For all $S$-algebra homomorphisms $i : R \rightarrow R'$, we have
$$F((A,C,\omega)\otimes_i R') = i(F(A,C,\omega)).$$
\item (Weight $k$ condition) Fix $k \in \mathbb{Z}$. For all $\lambda \in R^{\times}$, 
$$F(A,C,\lambda\cdot\omega) = \lambda^{-k}\cdot F(A,C,\omega).$$
\item (Regularity at cusps) For any positive integer $d|N$, letting $\text{Tate}(q) = \mathbb{G}_m/q^{\mathbb{Z}}$ denote the Tate curve over the $p$-adic completion of $R((q^{1/d}))$, and letting $C \subset \text{Tate}(q)[N]$ be any level $N$ structure, we have
$$F(\text{Tate}(q),C,du/u) \in R[[q^{1/d}]]$$
where $u$ is the canonical parameter on $\mathbb{G}_m$.
\end{enumerate}
 If $F$ satisfies conditions (1)-(2), we say it is a \emph{weak $p$-adic modular form over $S$ of level $N$}. If $F$ satisfies conditions (1)-(3), we say it is a \emph{$p$-adic modular form over $S$ of level $N$}. Denote the space of weak $p$-adic modular forms over $S$ of level $N$ and the space of $p$-adic modular forms over $S$ of level $N$ by $\tilde{M}_k^{\text{$p$-adic}}(\Gamma_0(N))$ and $M_k^{\text{$p$-adic}}(\Gamma_0(N))$, respectively. Note that $M_k^{\text{$p$-adic}}(\Gamma_0(N)) \subset \tilde M_k^{\text{$p$-adic}}(\Gamma_0(N))$.
\end{definition}
Let $\text{Tate}(q)$ be the Tate curve over the $p$-adic completion of $S((q))$. If $F \in \tilde{M}_k^{\text{$p$-adic}}(\Gamma_0(N))$, one defines the $q$-expansion (at infinity) of $F$ as $F(q) := F(\text{Tate}(q),\mu_N,du/u) \in S[[q]]$, which defines a \emph{$q$-expansion map} $F \mapsto F(q)$. The \emph{$q$-expansion principle} (see \cite[Theorem I.3.1]{Gouvea1988} or \cite{Katz1975}) says that the $q$-expansion map is injective for $F \in M_k^{\text{$p$-adic}}(\Gamma_0(N))$.

From now on, let $N$ denote the minimal level of $F$ (i.e. the smallest $N$ such that $F \in \tilde{M}_k^{\text{$p$-adic}}(\Gamma_0(N))$). For any positive integer $N'$ such that $N|N'$, we can define
$$[N'/N]^*F(A,C,\omega) := F(A,C[N],\omega)$$
so that $[N'/N]^*F \in \tilde{M}_k^{\text{$p$-adic}}(\Gamma_0(N'))$. When the larger level $N'$ is clear from context, we will often abuse notation and simply view $F \in \tilde{M}_k^{\text{$p$-adic}}(\Gamma_0(N'))$ by identifying $F$ and $[N'/N]^*F$.

We now fix $N^{\#}\in \mathbb{Z}_{>0}$ such that $N|N^{\#}$, so that we can view $F \in \tilde{M}_k^{\text{$p$-adic}}(\Gamma_0(N^{\#}))$, and further suppose $\ell^2|N^{\#}$ where $\ell$ is a prime (not necessarily different from $p$). Take the base ring $S = \mathcal{O}_{\mathbb{C}_p}$. Then the operator on $\tilde{M}_k^{\text{$p$-adic}}(\Gamma_0(N^{\#}))$ given on $q$-expansions by
$$F(q) \mapsto F(q^{\ell})$$
has a moduli-theoretic interpretation given by ``dividing by $\ell$-level structure''. That is, we have an operation on test triples $(A,C,\omega)$ defined over $p$-adic $\mathcal{O}_{\mathbb{C}_p}$-algebras $R$ given by
$$V_{\ell}(A,C,\omega) = (A/C[\ell],\pi(C),\check{\pi}^*\omega)$$
where $\pi : A \rightarrow A/C[\ell]$ is the canonical projection and $\check{\pi} : A/C[\ell] \rightarrow A$ is its dual isogeny.

Thus $V_{\ell}$ induces a form $V_{\ell}^*F \in \tilde{M}_k^{\text{$p$-adic}}(\Gamma_0(N^{\#}))$ defined by
$$V_{\ell}^*F(A,C,\omega) := F(V_{\ell}(A,C,\omega)).$$
For the Tate curve test triple $(\text{Tate}(q),\mu_{N^{\#}},du/u)$, one sees that $(\mu_{N^{\#}})[\ell] = \mu_{\ell}$ and $\pi : \text{Tate}(q) \rightarrow \text{Tate}(q^{\ell})$. Since $\pi : \hat{\mathbb{G}}_m = \widehat{\text{Tate}(q)} \rightarrow \widehat{\text{Tate}(q^{\ell})} = \hat{\mathbb{G}}_m$ is multiplication by $\ell$, we have $\pi^*du/u = \ell\cdot du/u$, and so  $\check{\pi}^*du/u = du/u$. Thus one sees that $V_{\ell}$ acts on $q$-expansions by
$$V_{\ell}^*F(q) = V_{\ell}^*F(\text{Tate}(q),\mu_{N^{\#}},du/u) = F(\text{Tate}(q^{\ell}),\mu_{N^{\#}/\ell},du/u) = F(q^{\ell}).$$
If $F \in M_k^{\text{$p$-adic}}(\Gamma_0(N^{\#}))$, then $V_{\ell}^*F \in M_k^{\text{$p$-adic}}(\Gamma_0(N^{\#}))$, and the $q$-expansion principle then implies that $V_{\ell}^*F$ is the unique $p$-adic modular form of level $N^{\#}$ with $q$-expansion $F(q^{\ell})$.
 
\subsection{Stabilization operators}\label{sec:stabilization}In this section, we define the ``stabilization operators'' alluded to in \S \ref{sec:idea-argument} as operations on rules on the moduli space of isomorphism classes of test triples. Let $F \in \tilde{M}_k^{\text{$p$-adic}}(\Gamma_0(N))$ and henceforth suppose $N$ is the \emph{minimal} level of $F$. View $F \in \tilde{M}_k^{\text{$p$-adic}}(\Gamma_0(N^{\#}))$, and let $a_{\ell}(F)$ denote the coefficient of the $q^{\ell}$ term in the $q$-expansion $F(q)$. Then up to permutation there is a unique pair of numbers $(\alpha_{\ell}(F), \beta_{\ell}(F)) \in \mathbb{C}_p^2$ such that $\alpha_{\ell}(F) + \beta_{\ell}(F) = a_{\ell}(F)$, $\alpha_{\ell}(F)\beta_{\ell}(F) = \ell^{k-1}$. We henceforth fix an ordered pair $(\alpha_{\ell}(F),\beta_{\ell}(F))$. 

\begin{definition}\label{stabilization}When $\ell\nmid N$, we define the \emph{$(\ell)^+$-stabilization of $F$} as
\begin{equation}\label{eq:stabilizationmoduli-1}
F^{(\ell)^+} = F - \beta_{\ell}(F)V_{\ell}^*F,
\end{equation}
the \emph{$(\ell)^-$-stabilization of $F$} as
\begin{equation}\label{eq:stabilizationmoduli-0}
F^{(\ell)^-} = F - \alpha_{\ell}(F)V_{\ell}^*F,
\end{equation}
and the \emph{$(\ell)^0$-stabilization for $F$} as
\begin{equation}\label{eq:stabilizationmoduli1}
F^{(\ell)^0} = F - a_{\ell}(F)V_{\ell}^*F + \ell^{k-1}V_{\ell}^*V_{\ell}^*F.
\end{equation}
We have $F^{(\ell)^{*}}\in M_k^{\text{$p$-adic}}(\Gamma_0(N^{\#}))$ for $* \in \{+,-,0\}$.

Observe that on $q$-expansions, we have
$$F^{(\ell)^+}(q) := F(q) - \beta_{\ell}(F)F(q^{\ell}),$$
$$F^{(\ell)^-}(q) := F(q) - \alpha_{\ell}(F)F(q^{\ell}),$$
$$F^{(\ell)^0}(q) := F(q) - a_{\ell}(F)F(q^{\ell}) + \ell^{k-1}F(q^{\ell^2}).$$
It follows that if $F$ is a $T_n$-eigenform where $\ell\nmid n$, then $F^{(\ell)^{*}}$ is still an eigenform for $T_n$. If $F$ is a $T_{\ell}$-eigenform, one verifies by direct computation that $a_{\ell}(F^{(\ell)^+}) = \alpha_{\ell}(F)$, $a_{\ell}(F^{(\ell)^-}) = \beta_{\ell}(F)$, and $a_{\ell}(F^{(\ell)^0}) = 0$. 

When $\ell|N$, we define the \emph{$(\ell)^0$-stabilization of $F$} as
\begin{equation}\label{eq:stabilizationmoduli2}
F^{(\ell)^0} = F - a_{\ell}(F)V_{\ell}^*F.
\end{equation}
Again, we have $F^{(\ell)^0} \in M_k^{\text{$p$-adic}}(\Gamma_0(N^{\#}))$. On $q$-expansions, we have
$$F^{(\ell)^0}(q) := F(q) - a_{\ell}(F)F(q^{\ell}).$$
It follows that if $F$ is a $U_n$-eigenform where $\ell\nmid n$, then $F^{(\ell)^0}$ is still an eigenform for $U_n$. If $F$ is a $U_{\ell}$-eigenform, one verifies by direct computation that $a_{\ell}(F^{(\ell)^0}) = 0$. 

Note that for $\ell_1 \neq \ell_2$, the stabilization operators $F \mapsto F^{(\ell_1)^{*}}$ and $F\mapsto F^{(\ell_2)^{*}}$ commute. Then for pairwise coprime integers with prime factorizations $N_+ = \prod_i \ell_i^{e_i}$, $N_- = \prod_j \ell^{e_j}$, $N_0 = \prod_m \ell_m^{e_m}$, we define the \emph{$(N_+,N_-,N_0)$-stabilization of $F$} as
$$F^{(N_+,N_-,N_0)} := F^{\prod_i (\ell_i)^+\prod_j (\ell_j)^-\prod_m (\ell_m)^0}.$$ 
\end{definition}

\subsection{Stabilization operators at CM points}\label{sec:stabilizationCM}
Let $K$ be an imaginary quadratic field satisfying the Heegner hypothesis with respect to $N^{\#}$. Assume that $p$ splits in $K$, and let $\mathfrak{p}$ be prime above $p$ determined by the embedding $K \subset \mathbb{C}_p$. Let $\mathfrak{N}^{\#} \subset \mathcal{O}_K$ be a fixed ideal such that $\mathcal{O}/\mathfrak{N}^{\#} = \mathbb{Z}/N^{\#}$, and if $p|N^{\#}$, we assume that $\mathfrak{p}|\mathfrak{N}^{\#}$. Let $A/\mathcal{O}_{\mathbb{C}_p}$ be an elliptic curve with CM by $\mathcal{O}_K$. By the theory of complex multiplication and Deuring's theorem, $(A,A[\mathfrak{N}^{\#}],\omega)$ is an ordinary test triple over $\mathcal{O}_{\mathbb{C}_p}$.

A crucial observation is that at an ordinary CM test triple $(A,A[\mathfrak{N}^{\#}],\omega)$, one can express $V_{\ell}(A,A[\mathfrak{N}^{\#}],\omega)$ and thus $(\ell)$-stabilization operators in terms of the action of $\mathcal{C}\ell(\mathcal{O}_K)$ on $A$ coming from Shimura's reciprocity law. First we recall the Shimura action: given an ideal $\mathfrak{a} \subset \mathcal{O}_K$, we define $A_{\mathfrak{a}} = A/A[\mathfrak{a}]$, an elliptic curve over $\mathcal{O}_{\mathbb{C}_p}$ which has CM by $\mathcal{O}_K$, whose isomorphism class depends only on the ideal class of $\mathfrak{a}$.
Let $\phi_{\mathfrak{a}} : A \rightarrow A_{\mathfrak{a}}$ denote the canonical projection. Note that there is an induced action of prime-to-$\mathfrak{N}^{\#}$ integral ideals $\mathfrak{a}\subset\mathcal{O}_K$ on the set of triples $(A,A[\mathfrak{N}^{\#}],\omega)$ given by  of isomorphism classes $[(A,A[\mathfrak{N}^{\#}],\omega)]$, given by
$$\mathfrak{a}\star(A,A[\mathfrak{N}^{\#}],\omega) = (A_{\mathfrak{a}},A_{\mathfrak{a}}[\mathfrak{N}^{\#}],\omega_{\mathfrak{a}})$$
where $\omega_{\mathfrak{a}}\in\Omega_{A_{\mathfrak{a}}/\mathbb{C}_p}^1$ is the unique differential such that $\phi_{\mathfrak{a}}^*\omega_{\mathfrak{a}} = \omega$. Note that this action descends to an action on the set of isomorphism classes of triples $[(A,A[\mathfrak{N}^{\#}],\omega)]$ given by
$\mathfrak{a}\star[(A,A[\mathfrak{N}^{\#}],\omega)] = [\mathfrak{a}\star(A,A[\mathfrak{N}^{\#}],\omega)]$. Letting $\mathfrak{N} = (\mathfrak{N}^{\#},N)$, also note that for any $\mathfrak{N}'\subset \mathcal{O}_K$ with norm $N'$ and $\mathfrak{N}|\mathfrak{N}'|N^{\#}$, the Shimura reciprocity law also induces an action of prime-to-$\mathfrak{N}'$ integral ideals on CM test triples and isomorphism classes of ordinary CM test triples of level $N'$.

The following calculation relates the values of $V_{\ell}$, $F^{(\ell)}$ and $F$ at CM test triples.

\begin{lemma}\label{lemma-CMstabcalc}For a prime $\ell$, let $v|\mathfrak{N}^{\#}$ be the corresponding prime ideal of $\mathcal{O}_K$ above it, let $\overline{v}$ denote the prime ideal which is the complex conjugate of $v$, and let $\mathfrak{a}\subset \mathcal{O}_K$ be an ideal prime to $\mathfrak{N}^{\#}$. Then for any $\omega \in \Omega_{A/\mathcal{O}_{\mathbb{C}_p}}^1$, we have
\begin{equation}\label{eq:CMstabilizationformula-1}
[V_{\ell}(\mathfrak{a}\overline{\mathfrak{N}^{\#}}\star(A,A[\mathfrak{N}^{\#}],\omega))] = [\overline{v}^{-1}\mathfrak{a}\overline{\mathfrak{N}^{\#}}\star(A,A[\mathfrak{N}^{\#}v^{-1}],\omega)]
\end{equation}
and 
\begin{equation}\label{eq:CMstabilizationformula0}
[V_{\ell}(V_{\ell}(\mathfrak{a}\overline{\mathfrak{N}^{\#}}\star(A,A[\mathfrak{N}^{\#}],\omega)))] = [\overline{v}^{-2}\mathfrak{a}\overline{\mathfrak{N}^{\#}}\star(A,A[\mathfrak{N}^{\#}v^{-2}],\omega)].
\end{equation}
As a consequence, if $F \in \tilde{M}_k^{\text{$p$-adic}}(\Gamma_0(N^{\#}))$, when $\ell\nmid N$ we have
\begin{equation}\label{eq:CMstabilizationformula1-1}
\begin{split}&F^{(\ell)^+}(\mathfrak{a}\overline{\mathfrak{N}^{\#}}\star(A,A[\mathfrak{N}^{\#}],\omega)) \\
&= F(\mathfrak{a}\overline{\mathfrak{N}^{\#}}\star(A,A[\mathfrak{N}^{\#}],\omega)) - \beta_{\ell}(F)F(\overline{v}^{-1}\mathfrak{a}\overline{\mathfrak{N}^{\#}}\star(A,A[\mathfrak{N}^{\#}],\omega)),
\end{split}
\end{equation}
\begin{equation}\label{eq:CMstabilizationformula1-0}
\begin{split}&F^{(\ell)^-}(\mathfrak{a}\overline{\mathfrak{N}^{\#}}\star(A,A[\mathfrak{N}^{\#}],\omega)) \\
&= F(\mathfrak{a}\overline{\mathfrak{N}^{\#}}\star(A,A[\mathfrak{N}^{\#}],\omega)) - \alpha_{\ell}(F)F(\overline{v}^{-1}\mathfrak{a}\overline{\mathfrak{N}^{\#}}\star(A,A[\mathfrak{N}^{\#}],\omega)),
\end{split}
\end{equation}
\begin{equation}\label{eq:CMstabilizationformula1}
\begin{split}&F^{(\ell)^0}(\mathfrak{a}\overline{\mathfrak{N}^{\#}}\star(A,A[\mathfrak{N}^{\#}],\omega)) \\
&= F(\mathfrak{a}\overline{\mathfrak{N}^{\#}}\star(A,A[\mathfrak{N}^{\#}],\omega)) - a_{\ell}(F)F(\overline{v}^{-1}\mathfrak{a}\overline{\mathfrak{N}^{\#}}\star(A,A[\mathfrak{N}^{\#}],\omega)) + \ell^{k-1}F(\overline{v}^{-2}\mathfrak{a}\overline{\mathfrak{N}^{\#}}\star(A,A[\mathfrak{N}^{\#}],\omega)),
\end{split}
\end{equation}
and when $\ell|N$,
\begin{equation}\label{eq:CMstabilizationformula2}
\begin{split}&F^{(\ell)^0}(\mathfrak{a}\overline{\mathfrak{N}^{\#}}\star(A,A[\mathfrak{N}^{\#}],\omega)) = F(\mathfrak{a}\overline{\mathfrak{N}^{\#}}\star(A,A[\mathfrak{N}^{\#}],\omega)) - a_{\ell}(F)F(\overline{v}^{-1}\mathfrak{a}\overline{\mathfrak{N}^{\#}}\star(A,A[\mathfrak{N}^{\#}],\omega)).
\end{split}
\end{equation}
\end{lemma}

\begin{proof}
Note that $(A_{\mathfrak{a}\overline{\mathfrak{N}^{\#}}}[\mathfrak{N}^{\#}])[\ell] = A_{\mathfrak{a}\overline{\mathfrak{N}^{\#}}}[v]$. Hence 
\begin{align*}[V_{\ell}(\mathfrak{a}\overline{\mathfrak{N}^{\#}}\star(A,A[\mathfrak{N}^{\#}],\omega))]
&= [\mathfrak{a}\overline{\mathfrak{N}^{\#}}\star V_{\ell}(A,A[\mathfrak{N}^{\#}],\omega)]\\
&= [\mathfrak{a}\overline{\mathfrak{N}^{\#}}\star(A_v,A_v[\mathfrak{N}^{\#}v^{-1}],\check{\phi}_v^*\omega)]\\
&= [\overline{v}^{-1}\mathfrak{a}\overline{\mathfrak{N}^{\#}}\star(A_{v\overline{v}},A_{v\overline{v}}[\mathfrak{N}^{\#}v^{-1}],(\check{\phi}_v^*\omega)_{\overline{v}})]\\
&=  [\overline{v}^{-1}\mathfrak{a}\overline{\mathfrak{N}^{\#}}\star(A_{(\ell)},A_{(\ell)}[\mathfrak{N}^{\#}v^{-1}],(\check{\phi}_{v}^*\omega)_{\overline{v}})]\\
&= [\overline{v}^{-1}\mathfrak{a}\overline{\mathfrak{N}^{\#}}\star(A,A[\mathfrak{N}^{\#}v^{-1}],\omega)]
\end{align*}
where the last equality, and hence (\ref{eq:CMstabilizationformula-1}) follows, once we prove the following.

\begin{lemma}Under the canonical isomorphism $i : A_{(\ell)} \xrightarrow{\sim} A$ sending an equivalence class $x + A[\ell] \in A_{(\ell)}$ to $[\ell]x$, where $[\ell] : A \rightarrow A$ denotes multiplication by $\ell$ in the group law, we have 
\begin{equation}\label{isog}(\check{\phi}_v^*\omega)_{\overline{v}} = i^*\omega.
\end{equation}
\end{lemma}
\begin{proof} By definition of $\omega_{\overline{v}}$ for a given differential $\omega$, (\ref{isog}) is equivalent to the identity
$$\check{\phi}_v^*\omega = \phi_{\overline{v}}^*(i^*\omega) = (i\circ \phi_{\overline{v}})^*\omega.$$
To show this, it suffices to establish the equality
$$\check{\phi}_v = i\circ \phi_{\overline{v}}$$
of isogenies $A_v \rightarrow A$. Since $\phi_{\overline{v}}\circ\phi_v = \phi_{(\ell)} = A \rightarrow A_{(\ell)}$, we have
$$i\circ \phi_{\overline{v}}\circ \phi_v  = i \circ \phi_{(\ell)} : A \xrightarrow{\phi_{(\ell)}} A_{(\ell)} \underset{\sim}{\xrightarrow{i}} A$$
where the first arrow maps $x\mapsto x + A[\ell]$, and the second arrow maps $x + A[\ell] \mapsto [\ell]x$. Hence this composition is in fact just the multiplication by $\ell$ map $[\ell]$. 
Hence $i\circ \phi_{\overline{v}}$ is the dual isogeny of $\phi_v$, i.e. $\check{\phi}_v = i\circ\phi_{\overline{v}}$, and the lemma follows.
\end{proof}

The identity (\ref{eq:CMstabilizationformula0}) follows by the same argument as above, replacing $\mathfrak{N}^{\#}$ with $\mathfrak{N}^{\#}v^{-1}$. Viewing $F$ as a form of level $N^{\#}$ and using (\ref{eq:CMstabilizationformula-1}) and (\ref{eq:CMstabilizationformula0}), then (\ref{eq:CMstabilizationformula1-1}),  (\ref{eq:CMstabilizationformula1-0}), (\ref{eq:CMstabilizationformula1}) and (\ref{eq:CMstabilizationformula2}) follow from (\ref{eq:stabilizationmoduli-1}), (\ref{eq:stabilizationmoduli-0}), (\ref{eq:stabilizationmoduli1}) and (\ref{eq:stabilizationmoduli2}), respectively.
\end{proof}

Finally, we relate the CM period sum of $F^{(\ell)^*}$ for $\* \in \{+,-,0\}$ to that of $F$ by showing that they differ by an Euler factor at $\ell$ associated with $F\otimes\chi^{-1}$. This calculation will be used in the proof of Theorem \ref{thm:genmaincongruence} to relate the values at Heegner points of the formal logarithms $\log_{\omega_{F^{(\ell)}}}$ and $\log_{\omega_F}$ associated with $F^{(\ell)^*}$ and $F$. 

\begin{lemma}\label{lemma-calc}Suppose $F\in \tilde{M}_k^{\text{$p$-adic}}(\Gamma_0(N^{\#}))$, and let $\chi: \mathbb{A}_K^{\times} \rightarrow \mathbb{C}_p^{\times}$ be a $p$-adic Hecke character such $\chi$ is unramified (at all finite places of $K$), and $\chi_{\infty}(\alpha) = \alpha^k$ for any $\alpha \in K^{\times}$. Let $\{\mathfrak{a}\}$ be a full set of integral representatives of $\mathcal{C}\ell(\mathcal{O}_K)$ where each $\mathfrak{a}$ is prime to $\mathfrak{N}^{\#}$. If $\ell\nmid N$, we have
\begin{align*}\sum_{[\mathfrak{a}]\in\mathcal{C}\ell(\mathcal{O}_K)}\chi^{-1}(\mathfrak{a})&F^{(\ell)^+}(\mathfrak{a}\star(A,A[\mathfrak{N}^{\#}],\omega)) \\
&= \left(1-\beta_{\ell}(F)\chi^{-1}(\overline{v})\right)\sum_{[\mathfrak{a}]\in\mathcal{C}\ell(\mathcal{O}_K)}\chi^{-1}(\mathfrak{a})F(\mathfrak{a}\star(A,A[\mathfrak{N}^{\#}],\omega)),
\end{align*}
\begin{align*}\sum_{[\mathfrak{a}]\in\mathcal{C}\ell(\mathcal{O}_K)}\chi^{-1}(\mathfrak{a})&F^{(\ell)^-}(\mathfrak{a}\star(A,A[\mathfrak{N}^{\#}],\omega)) \\
&= \left(1-\alpha_{\ell}(F)\chi^{-1}(\overline{v})\right)\sum_{[\mathfrak{a}]\in\mathcal{C}\ell(\mathcal{O}_K)}\chi^{-1}(\mathfrak{a})F(\mathfrak{a}\star(A,A[\mathfrak{N}^{\#}],\omega)),
\end{align*}
\begin{align*}\sum_{[\mathfrak{a}]\in\mathcal{C}\ell(\mathcal{O}_K)}\chi^{-1}(\mathfrak{a})&F^{(\ell)^0}(\mathfrak{a}\star(A,A[\mathfrak{N}^{\#}],\omega))\\
&= \left(1-a_{\ell}(F)\chi^{-1}(\overline{v}) + \frac{\chi^{-2}(\overline{v})}{\ell}\right)\sum_{[\mathfrak{a}]\in\mathcal{C}\ell(\mathcal{O}_K)}\chi^{-1}(\mathfrak{a})F(\mathfrak{a}\star(A,A[\mathfrak{N}^{\#}],\omega))
\end{align*}
and if $\ell|N$, we have
\begin{align*}&\sum_{[\mathfrak{a}]\in\mathcal{C}\ell(\mathcal{O}_K)}\chi^{-1}(\mathfrak{a})F^{(\ell)^0}(\mathfrak{a}\star(A,A[\mathfrak{N}^{\#}],\omega)) \\
&= \left(1-a_{\ell}(F)\chi^{-1}(\overline{v})\right)\sum_{[\mathfrak{a}]\in\mathcal{C}\ell(\mathcal{O}_K)}\chi^{-1}(\mathfrak{a})F(\mathfrak{a}\star(A,A[\mathfrak{N}^{\#}],\omega)).
\end{align*}
\end{lemma}

\begin{proof}First note that by our assumptions on $\chi$, for any $G \in \tilde{M}_k^{\text{$p$-adic}}(\Gamma_0(N^{\#}))$, the quantity
$$\chi^{-1}(\mathfrak{a})G(\mathfrak{a}\star(A,A[\mathfrak{N}^{\#}],\omega))$$
depends only on the ideal class $[\mathfrak{a}]$ of $\mathfrak{a}$. Since $\{\mathfrak{a}\}$ of integral representatives of $\mathcal{C}\ell(\mathcal{O}_K)$, $\{\mathfrak{a}\overline{\mathfrak{N}^{\#}}\}$ is also a full set of integral representatives of $\mathcal{C}\ell(\mathcal{O}_K)$. By summing over $\mathcal{C}\ell(\mathcal{O}_K)$ and applying Lemma \ref{lemma-CMstabcalc}, we obtain
\begin{align*}\sum_{[\mathfrak{a}]\in\mathcal{C}\ell(\mathcal{O}_K)}\chi^{-1}(\mathfrak{a})F^{(\ell)^0}(\mathfrak{a}\star(A,A[\mathfrak{N}^{\#}],\omega)) = &\sum_{[\mathfrak{a}]\in\mathcal{C}\ell(\mathcal{O}_K)}\chi^{-1}(\mathfrak{a})F(\mathfrak{a}\star(A,A[\mathfrak{N}^{\#}],\omega))\\
- a_{\ell}(F)&\sum_{[\mathfrak{a}]\in\mathcal{C}\ell(\mathcal{O}_K)}\chi^{-1}(\mathfrak{a}\overline{\mathfrak{N}^{\#}})F(\overline{v}^{-1}\mathfrak{a}\overline{\mathfrak{N}^{\#}}\star(A,A[\mathfrak{N}^{\#}],\omega)) \\
- \frac{1}{\ell}&\sum_{[\mathfrak{a}]\in\mathcal{C}\ell(\mathcal{O}_K)}\chi^{-1}(\mathfrak{a}\overline{\mathfrak{N}^{\#}})F(\overline{v}^{-2}\mathfrak{a}\overline{\mathfrak{N}^{\#}}\star(A,A[\mathfrak{N}^{\#}],\omega)) \\
= \left(1-a_{\ell}(F)\chi^{-1}(\overline{v}) + \frac{\chi^{-2}(\overline{v})}{\ell}\right)&\sum_{[\mathfrak{a}]\in\mathcal{C}\ell(\mathcal{O}_K)}\chi^{-1}(\mathfrak{a})F(\mathfrak{a}\star(A,A[\mathfrak{N}^{\#}],\omega))
\end{align*}
when $\ell\nmid N$. Similarly, we obtain the other identities for $(\ell)^+$ and $(\ell)^-$-stabilization when $\ell\nmid N$, as well as the identity for $(\ell)^0$-stabilization when $\ell|N$. 
\end{proof}


\subsection{Coleman integration}\label{sec:coleman-integration}
In this section, we recall Liu--Zhang--Zhang's extension of Coleman's theorem on $p$-adic integration. We will use this theorem later in order to directly realize (a pullback of) the formal logarithm along the weight 2 newform $f \in S_2^{\text{new}}(\Gamma_0(N))$ as a rigid analytic function $F$ on the ordinary locus of $X_0(N)(\mathbb{C}_p)$ (viewed as a rigid analytic space) satisfying $\theta F = f$. 

First we recall the theorem of Liu--Zhang--Zhang, closely following the discussion preceding Proposition A.1 in \cite[Appendix A]{Liu2014}. Let $R \subset \mathbb{C}_p$ be a local field. Suppose $X$ is a quasi-projective scheme over $R$, $X^{\text{rig}} = X(\mathbb{C}_p)^{\text{rig}}$ is its rigid-analytification, and $U \subset X^{\mathrm{rig}}$ an affinoid domain with good reduction. 

\begin{definition}Let $X$ and $U$ be as above, and let $\omega$ be a closed rigid analytic 1-form on $U$. Suppose there exists a locally analytic function $F_{\omega}$ on $U$ as well as a Frobenius endomorphism $\phi$ of $U$ (i.e. an endomorphism reducing to an endomorphism induced by a power of Frobenius on the reduction of $U$) and  a polynomial $P(X) \in \mathbb{C}_p[X]$ such that no root of $P(T)$ is a root of unity, satisfying
\begin{itemize}
\item $dF_{\omega} = \omega$;
\item $P(\phi^{*})F_{\omega}$ is rigid analytic;
\end{itemize}
and $F_{\omega}$ is uniquely determined by these conditions up to additive constant. We then call $F_{\omega}$ the \emph{Coleman primitive of $\omega$ on $U$}. It turns out that $F_{\omega}$, if it exists, is independent of the choice of $P(X)$ (\cite[Corollary 2.1b]{Coleman1985}). 
\end{definition}

Given an abelian variety $A$ over $R$ of dimension $d$, recall the formal logarithm defined as follows. Choosing a $\omega \in \Omega_{A/\mathbb{C}_p}^1$, the \emph{$p$-adic formal logarithm along $\omega$} is defined by formal integration 
$$\log_{\omega}(T) : = \int_0^T \omega$$
in a formal neighborhood $\hat{A}$ of the origin. Since $A(\mathbb{C}_p)$ is compact, we may extend by linearity to a map $\log_{\omega}: A(\mathbb{C}_p) \rightarrow \mathbb{C}_p$ (i.e., $\log_{\omega}(x) := \frac{1}{n}\log_{\omega}(nx)$ if $nx \in \hat{A}$). 

Liu--Zhang--Zhang prove the following extension of Coleman's theorem.

\begin{theorem}[See Proposition A.1 in \cite{Liu2014}]\label{ColemanExtension} Let $X$ and $U$ be as above. Let $A$ be an abelian variety over $R$ which has either totally degenerate reduction (i.e. after base changing to a finite extension of $R$, the connected component of the special fiber of the N\'{e}ron model of $A$ is isomorphic to $\mathbb{G}_m^d$), or potentially good reduction. For a morphism $\iota: X \rightarrow A$ and a differential form $\omega \in \Omega_{A/F}^1$, we have
\begin{enumerate}
\item $\iota^*\omega_{|U}$ admits a Coleman primitive on $U$, and in fact
\item $\iota^*\log_{\omega_{|U}}$ is a Coleman primitive of $\iota^*\omega_{|U}$ on $U$, where $\log_{\omega}: A(\mathbb{C}_p) \rightarrow \mathbb{C}_p$ is the $p$-adic formal logarithm along $\omega$.
\end{enumerate}
\end{theorem}

\subsection{The main congruence}\label{sec:proof-theorem}
Let $f \in M_2(\Gamma_0(N))$ and $g\in M_2(\Gamma_0(N'))$ be normalized eigenforms defined over the ring of integers of a number field with minimal levels $N$ and $N'$, respectively. Let $K$ be an imaginary quadratic field with Hilbert class field $H$, and suppose $K$ satisfies the Heegner hypothesis with respect to both $N$ and $N'$, with corresponding fixed choices of ideals $\mathfrak{N}, \mathfrak{N}' \subset \mathcal{O}_K$ such that $\mathcal{O}_K/\mathfrak{N} = \mathbb{Z}/N$, $\mathcal{O}_K/\mathfrak{N}' = \mathbb{Z}/N'$, and such that $\ell|(N,N')$ implies $(\ell,\mathfrak{N}) = (\ell,\mathfrak{N}')$; hence $\mathcal{O}_K/\mathrm{lcm}(\mathfrak{N},\mathfrak{N}') = \mathbb{Z}/\mathrm{lcm}(N,N')$.

Recall the moduli-theoretic interpretation of $X_0(N)$, in which points on $X_0(N)$ are identified with isomorphism classes $[(A,C)]$ of pairs $(A,C)$ consisting of an elliptic curve $A$ and a cyclic subgroup $C \subset A[N]$ of order $N$. Throughout this section, let $A/\mathcal{O}_{\mathbb{C}_p}$ be a fixed elliptic curve with CM by $\mathcal{O}_K$, and note that as in \S \ref{sec:stabilizationCM}, the Shimura reciprocity law induces an action of integral ideals prime to $\mathfrak{N}$ on $(A,A[\mathfrak{N}])$, which descends to an action of $\mathcal{C}\ell(\mathcal{O}_K)$ on $[(A,A[\mathfrak{N}])]$. Let $\chi : \text{Gal}(H/K) \rightarrow \overline{\mathbb{Q}}^{\times}$ be a character, and let $L$ be a finite extension of $K$ containing the Hecke eigenvalues of $f,g$, the values of $\chi$ and the field cut out by the kernel of $\chi$.  For any full set of prime-to-$\mathfrak{N}$ integral representatives $\{\mathfrak{a}\}$ of $\mathcal{C}\ell(\mathcal{O}_K)$, define the Heegner point on $J_0(N)$ attached to $\chi$ by
$$P(\chi) := \sum_{[\mathfrak{a}] \in \mathcal{C}\ell(\mathcal{O}_K)}\chi^{-1}(\mathfrak{a})([\mathfrak{a}\star(A,A[\mathfrak{N}])]-[\infty]) \in J_0(N)(H)\otimes_{\mathbb{Z}} L,$$
where $[\infty] \in X_0(N)(\mathbb{C}_p)$ denotes the cusp at infinity. Similarly, for any full set of prime-to-$\mathfrak{N}'$ integral representatives $\{\mathfrak{a}\}$ of $\mathcal{C}\ell(\mathcal{O}_K)$, define the Heegner point on $J_0(N')$ attached to $\chi$ by
$$P'(\chi) := \sum_{[\mathfrak{a}] \in \mathcal{C}\ell(\mathcal{O}_K)}\chi^{-1}(\mathfrak{a})([\mathfrak{a}\star(A,A[\mathfrak{N}'])] - [\infty']) \in J_0(N')(H)\otimes_{\mathbb{Z}} L,$$
where $[\infty'] \in X_0(N')(\mathbb{C}_p)$ denotes the cusp at infinity.

Let $\iota : X_0(N) \rightarrow J_0(N)$ denote the Abel-Jacobi map sending $[\infty] \mapsto 0$, and let $\iota' : X_0(N') \rightarrow J_0(N')$ denote the Abel-Jacobi map sending $[\infty'] \mapsto 0$. Let $A_f$ and $A_g$ be the abelian varieties over $\mathbb{Q}$ of $\GL_2$-type associated with $f$ and $g$. Fix modular parametrizations $\pi_f : J_0(N) \rightarrow A_f$ and $\pi_g : J_0(N') \rightarrow A_g$.  Let $P_f(\chi) := \pi_f(P(\chi))$ and $P_g(\chi) := \pi_g(P'(\chi))$. Letting 
$$\omega_f \in \Omega_{J_0(N)/\mathcal{O}_{\mathbb{C}_p}}^1 \;\text{such that}\; \iota^*\omega_f =  f(q)\cdot dq/q,$$
and 
$$\omega_g \in \Omega_{J_0(N')/\mathcal{O}_{\mathbb{C}_p}}^1 \;\text{such that}\; \iota'^{,*}\omega_g = g(q)\cdot dq/q,$$
we choose $\omega_{A_f} \in \Omega_{A_f/\mathbb{Q}}^1$ and $\omega_{A_g} \in \Omega_{A_g/\mathbb{Q}}^1$ such that $\pi_f^*\omega_{A_f} = \omega_f$ and $\pi_g^*\omega_{A_g} = \omega_g$. 

We define
$$\log_{\omega_f}P(\chi) := \sum_{[\mathfrak{a}]\in\mathcal{C}\ell(\mathcal{O}_K)}\chi^{-1}(\mathfrak{a})\log_{\omega_f}([\mathfrak{a}\star(A,A[\mathfrak{N}])] - [\infty]) \in L_p$$
and
$$\log_{\omega_g}P'(\chi) := \sum_{[\mathfrak{a}]\in\mathcal{C}\ell(\mathcal{O}_K)}\chi^{-1}(\mathfrak{a})\log_{\omega_g}([\mathfrak{a}\star(A,A[\mathfrak{N}'])] - [\infty']) \in L_p.$$
The fact that these are values in $L_p$ follows from the fact $P(\chi) \in J_0(N)(H)\otimes_{\mathbb{Z}}\overline{\mathbb{Q}}$ is in the $\chi$-isotypic component of $\text{Gal}(\overline{\mathbb{Q}}/K)$, and similarly for $P'(\chi)$. We similarly define $\log_{\omega_{A_f}}P_f(\chi) \in L_p$ and $\log_{\omega_{A_g}}P_g(\chi) \in L_p$, and note that  by functoriality of the $p$-adic logarithm, $\log_{\omega_f}P(\chi) = \log_{\omega_{A_f}}P_f(\chi)$ and $\log_{\omega_g}P'(\chi) = \log_{\omega_{A_g}}P_g(\chi)$. 

Let $\lambda$ be the prime of $\mathcal{O}_L$ above $p$ determined by the embedding $L\hookrightarrow \overline{\mathbb{Q}}_p$. We will now prove a generalization of Theorem \ref{thm:maincongruence} for general weight 2 forms.
\begin{theorem}\label{thm:genmaincongruence} In the setting and notations described above, suppose that the associated semisimple mod $\lambda^m$ representations $\bar{\rho}_f, \bar{\rho}_g : \Gal(\overline{\mathbb{Q}}/\mathbb{Q}) \rightarrow \GL_2(\mathcal{O}_{L_p}/\lambda^m)$ satisfy $\bar{\rho}_f \cong \bar{\rho}_g$. For each prime $\ell|NN'$, let $v|\mathfrak{NN}'$ be the corresponding prime above it. Then we have
\begin{align*}&\left(\prod_{\ell|pNN'/M,\ell\nmid N}\frac{\ell-a_{\ell}(f)\chi^{-1}(\overline{v}) + \chi^{-2}(\overline{v})}{\ell}\right) \left(\prod_{\ell|pNN'/M,\ell|N}\frac{\ell-a_{\ell}(f)\chi^{-1}(\overline{v})}{\ell}\right)\log_{\omega_{A_f}} P_f(\chi) \\
\equiv& \left(\prod_{\ell|pNN'/M,\ell\nmid N'}\frac{\ell-a_{\ell}(g)\chi^{-1}(\overline{v}) + \chi^{-2}(\overline{v})}{\ell}\right)\left(\prod_{\ell|pNN'/M,\ell|N'}\frac{\ell-a_{\ell}(g)\chi^{-1}(\overline{v})}{\ell}\right)\log_{\omega_{A_g}} P_g(\chi)\\
&\hspace{12.2cm}\pmod {\lambda^m\mathcal{O}_{L_p}},
\end{align*}
where
$$M = \prod_{\ell|(N,N'), a_{\ell}(f)\equiv a_{\ell}(g) \mod \lambda^m}\ell^{\ord_{\ell}(NN')}.$$
\end{theorem}

\begin{proof}[Proof of Theorem \ref{thm:genmaincongruence}]
  We first transfer all differentials and Heegner points on $J_0(N)$ and $J_0(N')$ to the Jacobian $J_0(N^{\#})$ of the modular curve $X_0(N^{\#})$, where $N^{\#} := \mathrm{lcm}_{\ell|NN'}(N,N',p^2,\ell^2)$. Note that for the newforms $f$ and $g$, the minimal levels of the stabilizations $f^{(\ell)}$ an $g^{(\ell)}$ divide $N^{\#}$, since if $\ell^2|N$ then $a_{\ell}(f) = 0$ and $f^{(\ell)} = f$, and similarly if $\ell^2|N'$ then $g^{(\ell)} = g$. By assumption, $K$ satisfies the Heegner hypothesis with respect to $N^{\#}$, and let $\mathfrak{N}^{\#} := \mathrm{lcm}_{v|\mathfrak{NN'}}(\mathfrak{N},\mathfrak{N}',\mathfrak{p}^2,v^2)$. For any full set of prime-to-$\mathfrak{N}^{\#}$ integral representatives $\{\mathfrak{a}\}$ of $\mathcal{C}\ell(\mathcal{O}_K)$, define 
$$P^{\#}(\chi) := \sum_{[\mathfrak{a}] \in \mathcal{C}\ell(\mathcal{O}_K)}\chi^{-1}(\mathfrak{a})([\mathfrak{a}\star(A,A[\mathfrak{N}^{\#}])] - [\infty^{\#}]) \in J_0(N^{\#})(H)\otimes_{\mathbb{Z}}L,$$
where $[\infty^{\#}] \in X_0(N^{\#})(\mathbb{C}_p)$ denotes the cusp at infinity. Letting $\pi^{\flat} : J_0(N^{\#}) \rightarrow J_0(N)$ and $\pi'^{,\flat} : J_0(N^{\#}) \rightarrow J_0(N')$ denote the natural projections, one sees that $\pi^{\flat}(P^{\#}(\chi)) = P(\chi)$ and that $\pi'^{,\flat}(P^{\#}(\chi)) = P'(\chi)$. Let $\iota^{\#} : X_0(N^{\#}) \rightarrow J_0(N^{\#})$ denote the Abel-Jacobi map sending $[\infty^{\#}] \mapsto 0$. Viewing $f$ and $g$ as having level $N^{\#}$, we define their associated differential forms by 
$$\omega_f^{\#} \in \Omega_{J_0(N^{\#})/\mathcal{O}_{\mathbb{C}_p}}^1 \;\text{such that}\; \iota^{\#,*}\omega_f^{\#} = f(q)\cdot dq/q \in \Omega_{X_0(N^{\#})/\mathcal{O}_{\mathbb{C}_p}}^1$$
and similarly define $\omega_g^{\#} \in \Omega_{J_0(N^{\#})/\mathcal{O}_{\mathbb{C}_p}}^1$. One sees that $\pi^{\flat,*}\omega_f = \omega_f^{\#}$ and $\pi'^{,\flat,*}\omega_g = \omega_g^{\#}$. Finally, define $$\log_{\omega_f^{\#}}P^{\#}(\chi) := \sum_{[\mathfrak{a}]\in\mathcal{C}\ell(\mathcal{O}_K)}\chi^{-1}(\mathfrak{a})\log_{\omega_f^{\#}}([\mathfrak{a}\star(A,A[\mathfrak{N}^{\#}])] - [\infty^{\#}]) \in L_p$$
and similarly for $\log_{\omega_g^{\#}}P^{\#}(\chi)$.

Let $N_0^{\#}$ denote the prime-to-$p$ part of $N^{\#}$. Let $\mathcal{X}$ denote the canonical smooth proper model of $X_0(N_0^{\#})$ over $\mathbb{Z}_p$, and let $\mathcal{X}_{\mathbb{F}_p}$ denote its special fiber. There is a natural reduction map $\mathrm{red} : X_0(N_0^{\#})(\mathbb{C}_p) = \mathcal{X}(\mathcal{O}_{\mathbb{C}_p}) \rightarrow \mathcal{X}_{\mathbb{F}_p}(\overline{\mathbb{F}}_p)$. Viewing $X_0(N_0^{\#})(\mathbb{C}_p)$ as a rigid analytic space, the inverse image in $X_0(N_0^{\#})(\mathbb{C}_p)$ of an element of the finite set of supersingular points in $\mathcal{X}_{\mathbb{F}_p}(\overline{\mathbb{F}}_p)$ is conformal to an open unit disc, and is referred to as a \emph{supersingular disc}. Let $\mathcal{D}_0$ denote the the affinoid domain of good reduction obtained by removing the finite union of supersingular discs from the rigid space $X_0(N_0^{\#})(\mathbb{C}_p)$. In the moduli-theoretic interpretation, $\mathcal{D}_0$ consists of points $[(A,C)]$ over $\mathcal{O}_{\mathbb{C}_p}$ of good reduction such that $A\otimes_{\mathcal{O}_{\mathbb{C}_p}}\overline{\mathbb{F}}_p$ is ordinary. The canonical projection $X_0(N^{\#}) \rightarrow X_0(N_0^{\#})$ has a \emph{rigid analytic} section on $\mathcal{D}_0$ given by ``increasing level $N_0^{\#}$ structure by the order $N^{\#}/N_0^{\#}$ canonical subgroup''. Namely given $[(A,C)] \in \mathcal{D}_0$, the section is defined by $[(A,C)] \mapsto [(A,C\times \hat{A}[N^{\#}/N_0^{\#}])]$. We identify $\mathcal{D}_0$ with its lift $\mathcal{D}$, which is called the \emph{ordinary locus} of $X_0(N^{\#})(\mathbb{C}_p)$; one sees from the above construction that $\mathcal{D}$ is an affinoid domain of good reduction. 

A $p$-adic modular form $F$ of weight 2 (as defined in \S \ref{sec:padicmodularforms}) can be equivalently viewed as a rigid analytic section of $(\Omega_{X_0(N^{\#})/\mathbb{C}_p}^1)_{|_{\mathcal{D}}}$ (viewed as an analytic sheaf). Under this identification, the exterior differential is given on $q$-expansions by $d = \theta\frac{dq}{q}$ where $\theta$ is the Atkin--Serre operator on $p$-adic modular forms acting via $q\frac{d}{dq}$ on $q$-expansions. Thus for each $j \in \mathbb{Z}_{\ge 0}$, $\theta^jF$ is a rigid analytic section of $(\Omega_{X_0(N^{\#})/\mathbb{C}_p}^{1+j})_{|_\mathcal{D}}$. The collection of $p$-adic modular forms $\theta^j(f^{(p)})$ varies $p$-adic continuously in $j \in \mathbb{Z}/(p-1)\times \mathbb{Z}_p$ (as one verifies on $q$-expansions), and so 
$$\theta^{-1}(f^{(p)}) := \lim_{j\rightarrow (-1,0)}\theta^j(f^{(p)})$$
is a rigid analytic function on $\mathcal{D}$ and a Coleman primitive for $\iota^{\#,*}\omega_{f^{(p)}}$ since $$d\theta^{-1}(f^{(p)}) = f^{(p)}(q)\cdot dq/q = \iota^{\#,*}\omega_{f^{(p)}}.$$ Also note that $\iota^{\#,*}\omega_f$ (restricted to $\mathcal{D}$) has a Coleman primitive $F_{\iota^{\#,*}\omega_f^{\#}}$ by part (1) of Theorem \ref{ColemanExtension} (applied to $R = \mathbb{Q}_p$, $X = X_0(N^{\#})$, $U = \mathcal{D}$ and $A = J_0(N^{\#})$), which we can (and do) choose to take the value 0 at $[\infty^{\#}]$. As a locally analytic function on $\mathcal{D}$, $F_{\iota^{\#,*}\omega_f^{\#}}$ can be viewed as an element of $\tilde{M}_0^{\text{$p$-adic}}(\Gamma_0(N^{\#}))$ (see Definition \ref{def:modforms}). By the moduli-theoretic definition of $(p)$-stabilization in terms of the operators $V_p$ defined in \S \ref{sec:stabilization}, we have $$d\theta^{-1}(f^{(p)}) = d(F_{\iota^{\#,*}\omega_f^{\#}})^{(p)},$$ and so $$\theta^{-1}(f^{(p)}) = (F_{\iota^{\#,*}\omega_f^{\#}})^{(p)}$$ by uniqueness of Coleman primitives. The same argument shows that $\theta^{-1}(g^{(p)}) = (F_{\iota^{\#,*}\omega_g^{\#}})^{(p)}$.

Since $\bar{\rho}_f \cong \bar{\rho}_g$, we have
$$\theta^j (f^{(pNN'/M)})(q) \equiv \theta^j (g^{(pNN'/M)})(q) \pmod {\lambda^m\mathcal{O}_{\mathbb{C}_p}}$$
for all $j \ge 0$. Letting $j \rightarrow (-1,0) \in \mathbb{Z}/(p-1)\times\mathbb{Z}_p$, we find that
$$\theta^{-1} (f^{(pNN'/M)})(q) \equiv \theta^{-1} (g^{(pNN'/M)})(q) \pmod {\lambda^m\mathcal{O}_{\mathbb{C}_p}}.$$
Let $N_0$ denote the prime-to-$p$ part of $NN'/M$. One sees directly from the description of stabilization operators on $q$-expansions that $\theta^{-1} (f^{(pNN'/M)})(q) = (\theta^{-1}(f^{(p)}))^{(N_0)}(q)$ and $\theta^{-1} (g^{(pNN'/M)})(q) = (\theta^{-1}(g^{(p)}))^{(N_0)}(q)$. Thus, the above congruence becomes
$$(\theta^{-1} (f^{(p)}))^{(N_0)}(q) \equiv (\theta^{-1} (g^{(p)}))^{(N_0)}(q) \pmod {\lambda^m\mathcal{O}_{\mathbb{C}_p}}.$$
Using the identities $\theta^{-1}(f^{(p)}) = (F_{\iota^{\#,*}\omega_f^{\#}})^{(p)}$ and $\theta^{-1}(g^{(p)}) = (F_{\iota^{\#,*}\omega_g^{\#}})^{(p)}$ and the equality of stabilization operators $(pN_0) = (pNN'/M)$, we have
$$(F_{\iota^{\#,*}\omega_f^{\#}})^{(pNN'/M)}(q) \equiv (F_{\iota^{\#,*}\omega_g^{\#}})^{(pNN'/M)}(q) \pmod {\lambda^m\mathcal{O}_{\mathbb{C}_p}}.$$
Thus, applying the $q$-expansion principle (i.e. the fact that the $q$-expansion map is injective), we have that
\begin{equation}
(F_{\iota^{\#,*}\omega_f^{\#}})^{(pNN'/M)} \equiv (F_{\iota^{\#,*}\omega_g^{\#}})^{(pNN'/M)} \pmod{\lambda^m\mathcal{O}_{\mathbb{C}_p}}
\end{equation}
as weight 0 $p$-adic modular forms on $\mathcal{D}$ over $\mathcal{O}_{\mathbb{C}_p}$. In particular, for an ordinary CM test triple $(A,A[\mathfrak{N}^{\#}],\omega)$, we have
\begin{equation}\label{eq:congruence}
(F_{\iota^{\#,*}\omega_f^{\#}})^{(pNN'/M)}(\mathfrak{a}\star(A,A[\mathfrak{N}^{\#}],\omega)) \equiv (F_{\iota^{\#,*}\omega_g^{\#}})^{(pNN'/M)}(\mathfrak{a}\star(A,A[\mathfrak{N}^{\#}],\omega)) \pmod{\lambda^m\mathcal{O}_{\mathbb{C}_p}}.
\end{equation}

Applying Lemma \ref{lemma-calc} inductively to $F_t = F_{\iota^{\#,*}\omega_f^{\#}}^{(\prod_{i = 1}^{r-t}\ell_i)}$ for $1 \le t \le r$ where $\prod_{i = 1}^r\ell_i$ is the square-free part of $pNN'/M$ (so that $F_0 = F_{\iota^{\#,*}\omega_f^{\#}}^{(pNN'/M)}$, $F_r = F_{\iota^{\#,*}\omega_f^{\#}}$ and $F_t^{(\ell_t)} = F_{t-1}$), and noting that $\theta F_{\iota^{\#,*}\omega_f^{\#}}(q) = f(q)$ implies $a_{\ell_t}(F_t) = a_{\ell_t}(f)/\ell_t$, we obtain, for any full set of prime-to-$\mathfrak{N}^{\#}$ integral representatives $\{\mathfrak{a}\}$ of $\mathcal{C}\ell(\mathcal{O}_K)$, 
\begin{align*}
&\sum_{[\mathfrak{a}]\in\mathcal{C}\ell(\mathcal{O}_K)}\chi^{-1}(\mathfrak{a})(F_{\iota^{\#,*}\omega_f^{\#}})^{(pNN'/M)}(\mathfrak{a}\star(A,A[\mathfrak{N}^{\#}],\omega)) \\
&= \left(\prod_{\ell|pNN'/M,\ell\nmid N}1-\frac{a_{\ell}(f)\chi^{-1}(\overline{v})}{\ell} + \frac{\chi^{-2}(\overline{v})}{\ell}\right)\left(\prod_{\ell|pNN'/M,\ell|N}1-\frac{a_{\ell}(f)\chi^{-1}(\overline{v})}{\ell}\right) \\
&\hspace{7cm}\cdot \sum_{[\mathfrak{a}]\in \mathcal{C}\ell(\mathcal{O}_K)}\chi^{-1}(\mathfrak{a})F_{\iota^{\#,*}\omega_f^{\#}}(\mathfrak{a}\star(A,A[\mathfrak{N}^{\#}],\omega))
\end{align*}
and similarly for $F_{\iota^{\#,*}\omega_g^{\#}}$. Thus by (\ref{eq:congruence}), we have
\begin{align*}&\left(\prod_{\ell|pNN'/M,\ell\nmid N}1-\frac{a_{\ell}(f)\chi^{-1}(\overline{v})}{\ell} + \frac{\chi^{-2}(\overline{v})}{\ell}\right)\left(\prod_{\ell|pNN'/M,\ell|N}1-\frac{a_{\ell}(f)\chi^{-1}(\overline{v})}{\ell}\right)\\ 
&\hspace{6cm}\cdot \sum_{[\mathfrak{a}]\in \mathcal{C}\ell(\mathcal{O}_K)}\chi^{-1}(\mathfrak{a})F_{\iota^{\#,*}\omega_f^{\#}}([\mathfrak{a}\star(A,A[\mathfrak{N}^{\#}])])\\
&\equiv \left(\prod_{\ell|pNN'/M,\ell\nmid N}1-\frac{a_{\ell}(g)\chi^{-1}(\overline{v})}{\ell} + \frac{\chi^{-2}(\overline{v})}{\ell}\right)\left(\prod_{\ell|pNN'/M,\ell|N}1-\frac{a_{\ell}(g)\chi^{-1}(\overline{v}) }{\ell}\right)\\ 
&\hspace{6cm}\cdot \sum_{[\mathfrak{a}]\in \mathcal{C}\ell(\mathcal{O}_K)}\chi^{-1}(\mathfrak{a})F_{\iota^{\#,*}\omega_g^{\#}}([\mathfrak{a}\star(A,A[\mathfrak{N}^{\#}])]) \pmod{\lambda^m\mathcal{O}_{\mathbb{C}_p}}. 
\end{align*}

By part (2) of Theorem \ref{ColemanExtension}, we have $F_{\iota^{\#,*}\omega_f^{\#}} = \iota^{\#,*}\log_{\omega_f^{\#}}$ and $F_{\iota^{\#,*}\omega_g^{\#}} = \iota^{\#,*}\log_{\omega_g^{\#}}$. Thus, the above congruence becomes 
\begin{align*}&\left(\prod_{\ell|pNN'/M,\ell\nmid N}1-\frac{a_{\ell}(f)\chi^{-1}(\overline{v})}{\ell} + \frac{\chi^{-2}(\overline{v})}{\ell}\right)\left(\prod_{\ell|pNN'/M,\ell|N}1-\frac{a_{\ell}(f)\chi^{-1}(\overline{v})}{\ell}\right)\log_{\omega_f^{\#}}P^{\#}(\chi)\\
&\equiv \left(\prod_{\ell|pNN'/M,\ell\nmid N}1-\frac{a_{\ell}(g)\chi^{-1}(\overline{v})}{\ell} + \frac{\chi^{-2}(\overline{v})}{\ell}\right)\left(\prod_{\ell|pNN'/M,\ell|N}1-\frac{a_{\ell}(g)\chi^{-1}(\overline{v}) }{\ell}\right)\log_{\omega_g^{\#}}P^{\#}(\chi)\\
&\hspace{13cm} \pmod{\lambda^m\mathcal{O}_{\mathbb{C}_p}}. 
\end{align*}
In fact, since both sides of this congruence belong to $L_p$ and $L_p\cap \mathcal{O}_{\mathbb{C}_p} = \mathcal{O}_{L_p}$, this congruence in fact holds mod $\lambda^m\mathcal{O}_{L_p}$. The theorem now follows from the functoriality of the $p$-adic logarithm:
$$\log_{\omega_f^{\#}}P^{\#}(\chi) = \log_{\pi^{\flat,*}\omega_f}P^{\#}(\chi) = \log_{\omega_f}P(\chi) = \log_{\pi_f^*\omega_{A_f}}P(\chi) = \log_{\omega_{A_f}}P_f(\chi)$$
and similarly $\log_{\omega_g^{\#}}P^{\#}(\chi) = \log_{\omega_{A_g}}P_g(\chi)$. 
\end{proof}

\begin{remark}The normalizations of $\omega_E$ and $\omega_{E'}$ in the statement of Theorem \ref{thm:maincongruence} \emph{a priori} imply that both sides of Theorem \ref{thm:maincongruence} are $p$-integral. This is because CM points are integrally defined by the theory of CM and the above proof shows that the rigid analytic function $\iota^{\#,*}\log_{\omega_{f^{(pNN'/M)}}}$ has integral $q$-expansion.

  Let $\omega_{\mathcal{E}}$ denote the canonical N\'{e}ron differential of $E$ (as we do in \S \ref{sec:proof-4}), and let $c \in \mathbb{Z}$ such that $\omega_{\mathcal{E}} = c\cdot\omega_E$. Note that the normalization of the $p$-adic formal logarithm $\log_{\omega_E}$ above differs by a factor of $c$ from that of the normalization $\log_E := \log_{\omega_{\mathcal{E}}}$.  So we know that $$\frac{|\tilde{E}^{\mathrm{ns}}(\mathbb{F}_p)|}{p\cdot c}\cdot\log_EP=\frac{|\tilde{E}^{\mathrm{ns}}(\mathbb{F}_p)|}{p}\cdot\log_{\omega_E}P$$ is $p$-integral. We remark this is compatible with the $p$-part of the BSD conjecture. In fact, the $p$-part of the BSD conjecture predicts that $P$ is divisible by $p^{\ord_p c}\cdot c_p(E)$ in $E(K)$ (see the conjectured formula (\ref{eq:BSDHeegner})) and so $\frac{|\tilde E^\mathrm{ns}(\mathbb{F}_p)|}{c}\cdot P$ lies in the formal group and hence $\frac{|\tilde E^\mathrm{ns}(\mathbb{F}_p)|}{c}\cdot \log_EP\in p \mathcal{O}_{K_p}$.
\end{remark}  

\begin{remark}Note that both sides of the congruence in the statement of Theorem \ref{thm:genmaincongruence} depend on the choices of appropriate $\mathfrak{N}, \mathfrak{N}'$ up to a sign $\pm 1$. In fact, for a rational prime $\ell|N$ (resp. $\ell|N'$), if we let $v = (\mathfrak{N},\ell)$ with complex conjugate prime ideal $\overline{v}$ (resp. $v' = (\mathfrak{N}',\ell)$ with complex conjugate prime ideal $\overline{v'}$), replacing $\mathfrak{N}$ with $\mathfrak{N}v^{-1}\overline{v}$ (resp. $\mathfrak{N}'$ with $\mathfrak{N}'v'^{-1}\overline{v'}$) amounts to performing an Atkin-Lehner involution on the Heegner point $P_f(\chi)$ (resp. $P_g(\chi)$), which amounts to multiplying the Heegner point by the local root number $w_{\ell}(A_f) \in \{\pm 1\}$ (resp. $w_{\ell}(A_g) \in \{\pm 1\}$). Our proof in fact shows that for whatever change we make in choice of $\mathfrak{N}$ (resp. $\mathfrak{N}'$), both sides are multiplied by the same sign $\pm 1$.
\end{remark}

\subsection{Proof of Theorem \ref{thm:maincongruence}}
It follows immediately from Theorem \ref{thm:genmaincongruence} by taking $\chi = \mathbf{1}$, $L = K$, and $f$ and $g$ to be associated with $E$ and $E'$. The Heegner points $P = P_f(\mathbf{1})$ and $P' = P_g(\mathbf{1})$ are defined up to sign and torsion depending on the choices of $\mathfrak{N}$ and $\mathfrak{N}'$ (see \cite{Gross1984}).  

\section{Goldfeld's conjecture for a general class of elliptic curves}\label{sec:goldf-conj-gener}

Our goal in this section is to prove Theorem \ref{thm:goldfeld}. Throughout this section we assume
\begin{center}
$E(\mathbb{Q})[2]=0$, or equivalently, $\Gal(\mathbb{Q}(E[2])/\mathbb{Q})\cong S_3$ or $\mathbb{Z}/3 \mathbb{Z}$.  
\end{center}
Notice that this assumption is mild and is satisfied by 100\% of all elliptic curves (when ordered by naive height).

\subsection{Explicit twists}
Now we restrict our attention to the following well-chosen set of twisting discriminants.

\begin{definition}\label{def:setS}
Given an imaginary quadratic field $K$ satisfying the Heegner hypothesis for $N$, we define the set $\mathcal{S}$ consisting of primes $\ell\nmid 2N$ such that
  \begin{enumerate}
   \item $\ell$ splits in $K$.
  \item $\Frob_\ell\in \Gal(\mathbb{Q}(E[2])/\mathbb{Q})$ has order 3.
  \end{enumerate}
We define $\mathcal{N}$ to be the set of all integers $d\equiv 1\pmod{4}$ such that $|d|$ is a square-free product of primes in $\mathcal{S}$. 
\end{definition}

\begin{remark}
    By Chebotarev's density theorem, the set of primes $\mathcal{S}$ has Dirichlet density $\frac{1}{6}=\frac{1}{2}\cdot\frac{1}{3}$ or $\frac{1}{3}=\frac{1}{2}\cdot\frac{2}{3}$ depending on $\Gal(\mathbb{Q}(E[2]/\mathbb{Q}))\cong S_3$ or $\mathbb{Z}/3 \mathbb{Z}$. In particular, there are infinitely many elements of $\mathcal{N}$ with $k$ prime factors for any fixed $k\ge1$.
\end{remark}

For $d\in \mathcal{N}$, we consider $E^{(d)}/\mathbb{Q}$, the quadratic twist of $E/\mathbb{Q}$ by $\mathbb{Q}(\sqrt{d})$. Since $d\equiv 1\pmod4$, we know that 2 is unramified in $\mathbb{Q}(\sqrt{d})$ and $E^{(d)}/\mathbb{Q}$ has conductor $Nd^2$.  Hence $K$ also satisfies the Heegner hypothesis for $Nd^2$. Let $P^{(d)}\in E^{(d)}(K)$ be the corresponding Heegner point. Since $$E[2]\cong E^{(d)}[2],$$ we can apply Theorem \ref{thm:maincongruence} to $E$ and $E^{(d)}$, $p=2$ and obtain the following theorem.

\begin{theorem}\label{thm:overK}
  Suppose $E/\mathbb{Q}$ is an elliptic curve with $E(\mathbb{Q})[2]=0$. Let $K$ be an imaginary quadratic field satisfying the Heegner hypothesis for $N$.
  Assume
  \begin{equation}
    \label{eq:star}\
2\text{ splits in } K \text{ and }\quad \frac{|\tilde E^\mathrm{ns}(\mathbb{F}_2)|\cdot\log_{\omega_E}(P)}{2}\not\equiv0\pmod{2}.     \tag{$\bigstar$}
  \end{equation} Then for any $d\in\mathcal{N}$:

  \begin{enumerate}
  \item \label{item:nonvanishing} We have $$\frac{|\tilde E^{(d),\mathrm{ns}}(\mathbb{F}_2)|\cdot\log_{\omega_{E^{(d)}}}(P^{(d)})}{2}\not\equiv0\pmod{2}.$$ In particular, $P^{(d)}\in E^{(d)}(K)$ is of infinite order and $E^{(d)}/K$ has both algebraic and analytic rank one.
  \item\label{item:rankpart} The rank part of the BSD conjecture is true for $E^{(d)}/\mathbb{Q}$ and $E^{(d\cdot d_K)}/\mathbb{Q}$. One of them has both algebraic and analytic rank one and the other has both algebraic and analytic rank zero.
  \item\label{item:rootnumber} $E^{(d)}/\mathbb{Q}$ (resp. $E^{(d\cdot d_K)}/\mathbb{Q}$) has the same rank as $E/\mathbb{Q}$ if and only if $\psi_d(-N)=1$ (resp. $\psi_d(-N)=-1$), where $\psi_d$ is the quadratic character associated to $\mathbb{Q}(\sqrt{d})/\mathbb{Q}$. 
  \end{enumerate}
\end{theorem}

\subsection{Proof of Theorem \ref{thm:overK}}\label{sec:proof-2}
  \begin{enumerate}
  \item We apply Theorem \ref{thm:maincongruence} to the two elliptic curves $E/\mathbb{Q}$ and $E^{(d)}/\mathbb{Q}$ and $p=2$. Let $\ell|Nd^2$ be a prime. Notice
    \begin{enumerate}
    \item if $\ell|| N$, $$a_\ell(E), a_\ell(E^{(d)})\in \{\pm1\},$$
    \item if $\ell^2|N$, $$a_\ell(E)=a_\ell(E^{(d)})=0,$$
    \item if $\ell\mid d$, we have $\ell\in\mathcal{S}$. Since $\Frob_\ell$ is order 3 on $E[2]$, we know that its trace $$a_\ell(E)\equiv 1\pmod{2}.$$ Since $\ell^2|Nd^2$, we know that $$a_\ell(E^{(d)})=0.$$
    \end{enumerate}
     It follows that $M=N^2$. The congruence formula in  Theorem \ref{thm:maincongruence} then reads: $$\frac{|\tilde E^\mathrm{ns}(\mathbb{F}_2)|}{2}\cdot \prod_{\ell\mid d}\frac{|\tilde E^\mathrm{ns}(\mathbb{F}_\ell)|}{\ell}\cdot\log_{\omega_E}P\equiv\frac{|\tilde E^{(d),\mathrm{ns}}(\mathbb{F}_2)|}{2}\cdot\prod_{\ell\mid d}\frac{|\tilde E^\mathrm{(d),ns}(\mathbb{F}_\ell)|}{\ell}\cdot \log_{\omega_{E^{(d)}}}P^{(d)}\pmod{2}.$$ Since $E$ has good reduction at $\ell\mid d$ and $\ell$ is odd, we have $$|\tilde E^\mathrm{ns}(\mathbb{F}_\ell)|=|E(\mathbb{F}_\ell)|=\ell+1-a_\ell(E)\equiv a_\ell(E)\equiv 1\pmod{2}.$$ Since $E^{(d)}$ has additive reduction at $\ell\mid d$ and $\ell$ is odd, we have $$|\tilde E^{(d),\mathrm{ns}}(\mathbb{F}_\ell)|=\ell\equiv 1\pmod{2}.$$ Therefore we obtain the congruence $$\frac{|\tilde E^\mathrm{ns}(\mathbb{F}_2)|\cdot \log_{\omega_E}P}{2}\equiv\frac{|\tilde E^{(d),\mathrm{ns}}(\mathbb{F}_2)|\cdot\log_{\omega_{E^{(d)}}}P^{(d)}}{2}\pmod{2}.$$ Assumption (\ref{eq:star}) says that the left-hand side is nonzero, hence the right-hand side is also nonzero. In particular, the Heegner point $P^{(d)}$ is of infinite order. The last assertion follows from the celebrated work of Gross--Zagier and Kolyvagin.
  \item Since $$L(E^{(d)}/K,s)=L(E^{(d)}/\mathbb{Q},s)\cdot L(E^{(d\cdot d_K)}/\mathbb{Q},s),$$ the sum of the analytic rank of $E^{(d)}/\mathbb{Q}$ and $E^{(d\cdot d_K)}/\mathbb{Q}$ is the equal to the analytic rank of $E^{(d)}/K$, which is one by the first part. Hence one of them has analytic rank one and the other has analytic rank zero. The remaining claims follow from Gross--Zagier and Kolyvagin.
  \item It is well-known that the global root numbers of quadratic twists are related by $$\varepsilon(E/\mathbb{Q})\cdot\varepsilon(E^{(d)}/\mathbb{Q})=\psi_d(-N).$$ It follows that $E^{(d)}/\mathbb{Q}$ and $E/\mathbb{Q}$ have the same  global root number if and only if $\psi_d(-N)$=1. Since the analytic ranks of $E^{(d)}/\mathbb{Q}$ and $E/\mathbb{Q}$ are at most one, the equality of global root numbers implies the equality of the analytic ranks.\qedhere
  \end{enumerate}

\subsection{Proof of Theorem \ref{thm:goldfeld}}\label{sec:proof-3}
This is a standard application of Ikehara's tauberian theorem (see, e.g., \cite[2.4]{Serre1976}). We include the argument for completeness. Since the set of primes $\mathcal{S}$ has Dirichlet density $\alpha=\frac{1}{6}$ or $\frac{1}{3}$ depending on $\Gal(\mathbb{Q}(E[2]/\mathbb{Q}))\cong S_3$ or $\mathbb{Z}/3 \mathbb{Z}$, we know that $$\sum_{\ell \in \mathcal{S}} \ell^{-s}\sim \alpha\cdot \log \frac{1}{s-1}, \quad s\rightarrow 1^+.$$ Then $$\log\left(\sum_{d\in \mathcal{N}} |d|^{-s}\right)=\log\left(\prod_{\ell\in \mathcal{S}}(1+\ell^{-s})\right)\sim \sum_{\ell\in \mathcal{S}}\ell^{-s}\sim\alpha\cdot\log\frac{1}{s-1}, \quad s\rightarrow 1^+.$$ Hence $$\sum_{d\in \mathcal{N}}|d|^{-s}=\frac{1}{(s-1)^{\alpha}}\cdot f(s)$$ for some function $f(s)$ holomorphic and nonzero when $\Re(s)\ge1$. It follows from Ikehara's tauberian theorem that $$\#\{d\in \mathcal{N}: |d|<X\}\sim c\cdot \frac{X}{\log^{1-\alpha} X},\quad X\rightarrow \infty$$ for some constant $c>0$. But by Theorem \ref{thm:overK} (\ref{item:rankpart}), we have for $r=0,1$, $$N_r(E, X)\ge\#\{d\in \mathcal{N}: |d|< X/ |d_K|\}.$$ The results then follow.

\section{The 2-part of the BSD conjecture over $K$}
\label{sec:proof-4}


\subsection{The strategy of the proof}\label{sec:strategy-proof} Let $E$ and $K$ be as in Theorem \ref{thm:2partBSD}. Under Assumption (\ref{eq:star}) and the assumption that $c_2(E)$ is odd, the Heegner point $P\in E(K)$ is \emph{indivisible by 2} (Lemma \ref{lem:indivisibleby2}), equivalently,  all the local Tamagawa numbers of $E$ are odd, and the 2-Selmer group $\Sel_2(E/K)$ has rank one (Corollary \ref{cor:BSDoverK1}). We are able to deduce that all the local Tamagawa numbers of $E^{(d)}$ are also odd (Lemma \ref{lem:tamagawa}), and $\Sel_2(E^{(d)}/K)$ also has rank one (Lemma \ref{lem:2selmer}). These are consequences of the primes in the well-chosen set $\mathcal{S}$ being \emph{silent} in the sense of Mazur--Rubin \cite{Mazur2015}. Notice that $\Sel_2(E^{(d)}/K)$ having rank one predicts that $E^{(d)}(K)$ has rank one and $\Sha(E^{(d)}/K)[2]$ is trivial, though it is not known in general how to show this directly (Remark \ref{rem:MazurRubin}). The advantage here is that we know \emph{a priori} from the mod 2 congruence that the Heegner point $P^{(d)}\in E^{(d)}(K)$ is also \emph{indivisible by 2}. Hence the prediction is indeed true and implies BSD(2) for $E^{(d)}/K$ (Corollary \ref{cor:BSDoverK2}).

Since the Iwasawa main conjecture is not known for $p=2$, the only known way to prove BSD(2) over $\mathbb{Q}$ is to compute the 2-part of both sides of (\ref{eq:bsdformula}) explicitly. We compute the 2-Selmer group $\Sel_2(E^{(d)}/\mathbb{Q})$ (Lemma \ref{lem:2selmerQ}) and compare this to a formula of Zhai \cite{Zhai2016} (based on modular symbols) for 2-part of algebraic $L$-values for rank zero twists. This allows us to deduce BSD(2) for the rank zero curve among $E^{(d)}$ and $E^{(d\cdot d_K)}$ (Lemma \ref{lem:rankzero2part}). Finally, BSD(2) for $E^{(d)}/K$ and BSD(2) for the rank zero curve together imply BSD(2) for the rank one curve among $E^{(d)}$ and $E^{(d\cdot d_K)}$.

\subsection{BSD(2) for $E/K$}
By the Gross--Zagier formula, the BSD conjecture for $E/K$ is equivalent to the equality (\cite[V.2.2]{Gross1986}) 
\begin{equation}\label{eq:BSDHeegner}u_K\cdot c_E\cdot \prod_{\ell\mid N}c_\ell(E)\cdot  |\Sha(E/K)|^{1/2}=[E(K): \mathbb{Z} P],
\end{equation}
where $u_K=|\mathcal{O}_K^\times/\{\pm1\}|$, $c_E$ is the Manin constant of $E/\mathbb{Q}$, $c_\ell(E)=[E(\mathbb{Q}_\ell):E^0(\mathbb{Q}_\ell)]$ is the local Tamagawa number of $E$ and $[E(K): \mathbb{Z}P]$ is the index of the Heegner point $P\in E(K)$. By Assumption (\ref{eq:star}) that 2 splits in $K$, we know $K\ne \mathbb{Q}(\sqrt{-1})$ or $\mathbb{Q}(\sqrt{-3})$, so $u_K=1$. Therefore the BSD conjecture for $E/K$ is equivalent to the equality
\begin{equation}
  \label{eq:grosszagier}
  \prod_{\ell\mid N}c_\ell(E)\cdot  |\Sha(E/K)|^{1/2}=\frac{[E(K): \mathbb{Z} P]}{c_E},
\end{equation}

\begin{lemma}\label{lem:indivisibleby2}
The right-hand side of (\ref{eq:grosszagier}) is a 2-adic unit.
\end{lemma}

\begin{proof}
Since $\mathbb{Q}(E[2])/\mathbb{Q}$ is an $S_3$ or $\mathbb{Z}/3 \mathbb{Z}$ extension, we know that the Galois representation $E[2]$ remains irreducible when restricted to any quadratic field, hence $E(K)[2]=0$.

Notice that the Manin constant $c_E$ is odd: it follows from \cite[Theorem A]{Abbes1996} when $E$ is good at 2, from \cite[p.270 (ii)]{Abbes1996} when $E$ is multiplicative at 2 since $c_2(E)$ is assumed to be odd, and by our extra assumption when $E$ is additive at 2.

Since $c_E$ is odd, we know that the right-hand side of (\ref{eq:grosszagier}) 2-adically integral. If it is not a 2-adic unit, then there exists some $Q\in E(K)$ such that $2 Q$ is an odd multiple of $P$. Let $\omega_\mathcal{E}$ be the N\'{e}ron differential of $E$ and let $\log_E:=\log_{\omega_\mathcal{E}}$. By the very definition of the Manin constant we have $c_E\cdot\omega_E=\omega_\mathcal{E}$ and $c_E\cdot\log_{\omega_E}=\log_E$. Hence up to a 2-adic unit, we have $$\frac{|\tilde E^\mathrm{ns}(\mathbb{F}_2)|\cdot\log_{\omega_E}P}{2}=\frac{|\tilde E^\mathrm{ns} (\mathbb{F}_2)|\cdot\log_EP}{2}=|\tilde E^\mathrm{ns} (\mathbb{F}_2)|\cdot\log_E(Q).$$ On the other hand, $c_2(E)\cdot |\tilde E^\mathrm{ns}(\mathbb{F}_2)|\cdot Q$ lies in the formal group $\hat E(2\mathcal{O}_{K_2})$ and $c_2(E)$ is assumed to be odd, we know that $$|\tilde E^\mathrm{ns} (\mathbb{F}_2)|\cdot\log_E(Q)\in 2 \mathcal{O}_{K_2},$$  which contradicts (\ref{eq:star}). So the right-hand side of (\ref{eq:grosszagier}) is a 2-adic unit.
\end{proof}

Since the left-hand side of (\ref{eq:grosszagier}) is a product of integers, Lemma \ref{lem:indivisibleby2} implies the following.

\begin{corollary}\label{cor:BSDoverK1} BSD(2) for $E/K$ is equivalent to that
\begin{center}
  all the local Tamagawa numbers $c_\ell(E)$ are odd and $\Sha(E/K)[2]=0$.  
\end{center}
\end{corollary}

\subsection{BSD(2) for $E^{(d)}/K$}

Let $d\in \mathcal{N}$.  The BSD conjecture for $E^{(d)}/K$ is equivalent to the equality
\begin{equation}
  \label{eq:grosszagier2}
  \prod_{\ell\mid Nd^2}c_\ell(E^{(d)})\cdot  |\Sha(E^{(d)}/K)|^{1/2}=\frac{[E^{(d)}(K): \mathbb{Z} P^{(d)}]}{c_{E^{(d)}}},
\end{equation}

\begin{lemma}\label{lem:tamagawa}
Assume BSD(2) is true for $E/K$. Then $c_\ell(E^{(d)})$ is odd for any $\ell\mid Nd^2$.
\end{lemma}

\begin{proof}
First consider $\ell\mid N$. Let $\mathcal{E}$ and $\mathcal{E}^{(d)}$ be the N\'{e}ron model over $\mathbb{Z}_\ell$ of $E$ and $E^{(d)}$ respectively. Notice that $E^{(d)}/\mathbb{Q}_p$ is the unramified quadratic twist of $E^{(d)}$. Since N\'{e}ron models commute with unramified base change, we know that the component groups $\Phi_\mathcal{E}$ and $\Phi_{\mathcal{E}^{(d)}}$ are quadratic twists of each other as $\Gal(\overline{\mathbb{F}}_\ell/\mathbb{F}_\ell)$-modules. In particular, $\Phi_\mathcal{E}[2]\cong\Phi_{\mathcal{E}^{(d)}}[2]$ as $\Gal(\overline{\mathbb{F}}_\ell/\mathbb{F}_\ell)$-modules and thus $$\Phi_\mathcal{E}(\mathbb{F}_\ell)[2]\cong\Phi_{\mathcal{E}^{(d)}}(\mathbb{F}_\ell)[2].$$ It follows that $c_\ell(E)$ and $c_\ell(E^{(d)})$ have the same parity.

  Next consider $\ell\mid d$. Since $E^{(d)}$ has additive reduction and $\ell$ is odd, thus we know that $$E^{(d)}(\mathbb{Q}_\ell)[2]\cong \Phi_{\mathcal E^{(d)}}(\mathbb{F}_\ell)[2].$$ Since $\ell\in \mathcal{S}$, $\Frob_\ell$ is assumed to have order 3 acting on $E^{(d)}[2]\cong E[2]$, we know that $E^{(d)}(\mathbb{Q}_\ell)[2]=0$. Hence $c_\ell(E^{(d)})$ is odd.
\end{proof}

\begin{lemma}\label{lem:indivisibleby2d}
Assume BSD(2) is true for $E/K$. The right-hand side of (\ref{eq:grosszagier2}) is a 2-adic unit.
\end{lemma}

\begin{proof}
  Since $E$ has no rational 2-torsion, we know that the Manin constants (with respect to both $X_0(N)$-parametrization and $X_1(N)$-parametrization) for all curves in the isogeny of $E$ have the same 2-adic valuation. The twisting argument of Stevens \cite[\S 5]{Stevens1989} shows that if the Manin constant $c_1$ for the $X_1(N)$-optimal curve in the isogeny class of $E$ is 1, then the Manin constant $c_1^{(d)}$ for the $X_1(N)$-optimal curve in the isogeny class of $E^{(d)}$ is also 1. The same twisting argument in fact shows that if $c_1$ is a 2-adic unit, then $c_1^{(d)}$ is also a 2-adic unit. Since $c_E$ is odd, we know that $c_1$ is odd, therefore $c_1^{(d)}$ is also odd. Since $E^{(d)}$ has no rational 2-torsion, it follows that the Manin constant $c_{E^{(d)}}$ is also odd.

Now using $c_2(E^{(d)})$ is odd  (by Lemma \ref{lem:tamagawa}) and $c_{E^{(d)}}$ is odd, and replacing $E$ by $E^{(d)}$ and replacing (\ref{eq:star}) by the conclusion of Theorem \ref{thm:overK}~(\ref{item:nonvanishing}), the same argument as in the proof of Lemma \ref{lem:indivisibleby2} shows that the right-hand side of (\ref{eq:grosszagier2}) is also a 2-adic unit.
\end{proof}

Again, since the left-hand side of (\ref{eq:grosszagier2}) is a product of integers, Lemma \ref{lem:indivisibleby2d} implies the following.

\begin{corollary}\label{cor:BSDoverK2}
   BSD(2) for $E^{(d)}/K$ is equivalent to that
\begin{center}
  all the local Tamagawa numbers $c_\ell(E^{(d)})$ are odd and $\Sha(E^{(d)}/K)[2]=0$.
\end{center}  
\end{corollary}

\subsection{2-Selmer groups over $K$} Now let us compare the 2-Selmer groups of $E/K$ and $E^{(d)}/K$.

\begin{lemma}\label{lem:2selmer}
Assume BSD(2) is true for $E/K$. The isomorphism of Galois representations $E[2]\cong E^{(d)}[2]$ induces an isomorphism of 2-Selmer groups $$\Sel_2(E/K)\cong \Sel_2(E^{(d)}/K).$$ In particular, $$\Sha(E^{(d)}/K)[2]=0.$$
\end{lemma}

\begin{proof}
The 2-Selmer group $\Sel_2(E/K)$ is defined by the local Kummer conditions $$\mathcal{L}_v(E/K)=\mathrm{im}\left(E(K_v)/2 E(K_v)\rightarrow H^1(K_v, E[2])\right).$$ Denote by $\mathcal{L}_v(E^{(d)}/K)$ the local Kummer conditions for $E^{(d)}/K$. It suffices to show that $\mathcal{L}_v(E/K)=\mathcal{L}_v(E^{(d)}/K)$ are the same at all places $v$ of $K$:
  \begin{enumerate}
  \item $v\mid\infty$: Since $v$ is complex, $H^1(K_v, E[2])=0$. So $\mathcal{L}_v(E/K)=\mathcal{L}_v(E^{(d)}/K)=0$.
  \item $v\mid d$: Suppose $v$ lies above $\ell \in\mathcal{S}$. Since $\Frob_\ell$ acts by order 3 on $E[2]$, we know that the unramified cohomology $$H^1_\mathrm{ur}(\mathbb{Q}_\ell, E[2])\cong E[2]/(\Frob_\ell-1)E[2]=0$$ (such $\ell$ is called \emph{silent} by Mazur--Rubin), and thus $\dim H^1(\mathbb{Q}_\ell,E[2])=2\dim H^1_\mathrm{ur}(\mathbb{Q}_\ell,E[2])=0$ (\cite[I.2.6]{Milne1986}). Since $\ell$ is split in $K$, it follows that $$H^1(K_v, E[2])\cong H^1(\mathbb{Q}_\ell, E[2])=0,$$   So $\mathcal{L}_v(E/K)=\mathcal{L}_v(E^{(d)}/K)=0$.
  \item $v\nmid d\infty$: By \cite[Lemma 2.9]{Mazur2010}, we have $$\mathcal{L}_v(E/K)\cap \mathcal{L}_v(E^{(d)}/K)=E_\mathbb{N}(K_v)/2E(K_v),$$ where $$E_\mathbb{N}(K_v)=\mathrm{im}\left(\mathbb{N}: E(L_v)\rightarrow E(K_v)\right)$$ is the image of the norm map induced from the quadratic extension $L_v=K_v(\sqrt{d})$ over $K_v$. To show that $\mathcal{L}_v(E/K)=\mathcal{L}_v(E^{(d)}/K)$, it suffices to show that $$E(K_v)/\mathbb{N}E(L_v)=0.$$ By local Tate duality, it suffices to show that $$H^1(\Gal(L_v/K_v), E(L_v))=0.$$ Notice that $K_v\cong \mathbb{Q}_\ell$ and $L_v/K_v$ is the unramified quadratic extension, we know that $$E(L_v)/E^0(L_v)\cong\Phi_{\mathcal{E}}(\mathbb{F}_{\ell^2}),$$ where $\Phi_\mathcal{E}$ is the component group of the N\'{e}ron model of $E$ over $\mathbb{Z}_\ell$. Let $c\in \Gal(\mathbb{F}_{\ell^2}/\mathbb{F}_\ell)$ be the order two automorphism, then $\Phi_\mathcal{E}(\mathbb{F}_{\ell^2})[2]^c=\Phi_\mathcal{E}(\mathbb{F}_\ell)[2]$. Since $c_\ell(E)$ is odd, it follows that $\Phi_\mathcal{E}(\mathbb{F}_{\ell^2})[2]^c=\Phi_\mathcal{E}(\mathbb{F}_\ell)[2]=0$. Since an order two automorphism on a nonzero $\mathbb{F}_2$-vector space must have a nonzero fixed vector, we know that $\Phi_\mathcal{E}(\mathbb{F}_{\ell^2})[2]=0$. Therefore $E(L_v)/E^0(L_v)$ has odd order. It remains to show that $$H^1(\Gal(L_v/K_v), E^0(L_v))=0,$$ which is true by Lang's theorem since $L_v/K_v$ is unramified (see \cite[Prop.~4.3]{Mazur1972}).\qedhere
  \end{enumerate}
\end{proof}

\subsection{Proof of Theorem \ref{thm:2partBSD} (\ref{item:2partK})} It follows immediately from Corollary \ref{cor:BSDoverK2}, Lemma \ref{lem:tamagawa} and Lemma \ref{lem:2selmer}.

\section{The 2-part of the BSD conjecture over $\mathbb{Q}$}
\label{sec:proof-5}

Let $E$ and $K$ be as in Theorem \ref{thm:2partBSD}. Let $d\in \mathcal{N}$.

\subsection{2-Selmer groups over $\mathbb{Q}$} Let us begin by comparing the 2-Selmer groups of $E/\mathbb{Q}$ and $E^{(d)}/\mathbb{Q}$.

\begin{lemma}\label{lem:2selmerQ} Let $\Delta(E)$ be the discriminant of a Weierstrass equation of $E/\mathbb{Q}$.
  \begin{enumerate}
  \item \label{item:1} If $\Delta(E)<0$, then $\Sel_2(E/\mathbb{Q})\cong\Sel_2(E^{(d)}/\mathbb{Q})$.
  \item \label{item:2} If $\Delta(E)>0$ and $d>0$, then $\Sel_2(E/\mathbb{Q})\cong\Sel_2(E^{(d)}/\mathbb{Q})$.
  \item \label{item:3} If $\Delta(E)>0$ and $d<0$, then $\dim_{\mathbb{F}_2} \Sel_2(E/\mathbb{Q})$ and $\dim_{\mathbb{F}_2}\Sel_2(E^{(d)}/\mathbb{Q})$ differ by 1.
  \end{enumerate}
\end{lemma}

\begin{proof}
  By the same proof as Lemma \ref{lem:2selmer}, we know that $\mathcal{L}_v(E/\mathbb{Q})=\mathcal{L}_v(E^{(d)}/\mathbb{Q})$ for any place $v\nmid\infty$ of $\mathbb{Q}$. The only issue is that the local condition at $\infty$ may differ for $E/\mathbb{Q}$ and $E^{(d)}/\mathbb{Q}$. By \cite[p.305]{Serre1972}, we have $\mathbb{Q}(\sqrt{\Delta(E)})\subseteq \mathbb{Q}(E[2])$. So complex conjugation acts nontrivially on $E[2]$ if and only if $\Delta(E)<0$. Hence  $$\dim_{\mathbb{F}_2} H^1(\Gal(\mathbb{C}/\mathbb{R}), E[2])=
  \begin{cases}
    0, & \Delta(E)<0, \\
    2, & \Delta(E)>0.
  \end{cases}
$$ The item (\ref{item:1}) follows immediately. When $\Delta(E)>0$,  $\mathcal{L}_\infty(E/\mathbb{Q})=E(\mathbb{R})/2 E(\mathbb{R})$ and $\mathcal{L}_\infty(E^{(d)}(\mathbb{R})=E^{(d)}(\mathbb{R})/2 E^{(d)}(\mathbb{R})$ define the same line in $H^1(\Gal(\mathbb{C}/\mathbb{R}),E[2])$ if and only if $d>0$. The item (\ref{item:2}) follows immediately and the item (\ref{item:3}) follows from a standard application of global duality (e.g., by \cite[Lemma 8.5]{LeHung2015}).
\end{proof}

We immediately obtain a more explicit description of the condition $\chi_d(-N)=1$ in Theorem \ref{thm:overK} (\ref{item:rootnumber}) under our extra assumption that $c_2(E)$ is odd.

\begin{corollary}\label{cor:disc} The following conditions are equivalent.
  \begin{enumerate}
  \item $E^{(d)}/\mathbb{Q}$ has the same rank as $E/\mathbb{Q}$.
  \item $\chi_d(-N)=1$, where $\chi_d$ is the quadratic character associated to $\mathbb{Q}(\sqrt{d})/\mathbb{Q}$.
  \item $\Delta(E)<0$, or $\Delta(E)>0$ and $d>0$.
  \end{enumerate}
\end{corollary}

\begin{proof}
  Since the parity conjecture for 2-Selmer groups of elliptic curves is known (\cite[Theorem 1.5]{Monsky1996}), we know that $E/\mathbb{Q}$ and $E^{(d)}/\mathbb{Q}$ has the same root number if and only if they have the same 2-Selmer rank. The result then follows from Lemma \ref{lem:2selmerQ} and Theorem \ref{thm:overK} (\ref{item:rootnumber}).
\end{proof}

\subsection{Rank zero twists}
Let $K$ be as in Theorem \ref{thm:2partBSD}.  We now verify BSD(2) for the rank zero twists. 

\begin{lemma}\label{lem:rankzero2part}
If BSD(2) is true for $E/\mathbb{Q}$ and $E^{(d_K)}/\mathbb{Q}$, then BSD(2) is true for  all twists $E^{(d)}/\mathbb{Q}$ and $E^{(d\cdot d_K)}/\mathbb{Q}$ of rank zero, where $d\in\mathcal{N}$ with $\chi_d(-N)=1$.
\end{lemma}

\begin{proof}
  Notice exactly one of $E/\mathbb{Q}$ and $E^{(d_K)}/\mathbb{Q}$ has rank zero.  Consider the case that $E/\mathbb{Q}$ has rank zero. Since all the local Tamagawa numbers $c_\ell(E)$ are odd and $\Sha(E/\mathbb{Q})[2]=0$, BSD(2) for $E/\mathbb{Q}$ implies that $$\frac{L(E/\mathbb{Q},1)}{\Omega(E/\mathbb{Q})}$$ is a 2-adic unit. Assume $\chi_d(-N)=1$. We know from Corollary \ref{cor:disc} that $\Delta(E)<0$, or $\Delta(E)>0$ and $d>0$. Under these conditions, it follows from \cite[Theorem 1.1, 1.3]{Zhai2016} that $$\frac{L(E^{(d)}/\mathbb{Q},1)}{\Omega(E^{(d)}/\mathbb{Q})}$$ is also a 2-adic unit (notice that the N\'{e}ron period $\Omega(E/\mathbb{Q})$ is twice of the real period when $\Delta(E)>0$).  Since all the local Tamagawa numbers $c_\ell(E^{(d)})$ are odd (Lemma \ref{lem:tamagawa}) and $\Sha(E^{(d)}/\mathbb{Q})[2]=0$ (Lemma \ref{cor:disc}, (\ref{item:1}, \ref{item:2})), we know that BSD(2) is true for $E^{(d)}/\mathbb{Q}$. By the same argument, if $E^{(d_K)}/\mathbb{Q}$  has rank zero and $\chi_d(-N)=1$, we know that BSD(2) is true for $E^{(d\cdot d_K)}/\mathbb{Q}$.
\end{proof}

\subsection{Proof of Theorem \ref{thm:2partBSD} (\ref{item:2partQ})}
Now we can finish the proof of Theorem  \ref{thm:2partBSD} (\ref{item:2partQ}). Because the abelian surface $E\times E^{(d_K)}/\mathbb{Q}$ is isogenous to the Weil restriction $\Res_{K/\mathbb{Q}} E$ and the validity of the BSD conjecture for abelian varieties is invariant under isogeny (\cite[I.7.3]{Milne2006}), we know that BSD(2) for $E/\mathbb{Q}$ and $E^{(d_K)}/\mathbb{Q}$ implies that BSD(2) is true for $E/K$.  Hence by Theorem \ref{thm:2partBSD} (\ref{item:2partQ}),  BSD(2) is true for $E^{(d)}/K$. By Lemma \ref{lem:rankzero2part}, BSD(2) is true for the rank zero curve among $E^{(d)}/\mathbb{Q}$ and $E^{(d\cdot d_K)}/\mathbb{Q}$ for $d\in \mathcal{N}$ such that $\chi_d(-N)=1$. Then again by the invariance of BSD(2) under isogeny, we know BSD(2) is also true for the other rank one curve among $E^{(d)}/\mathbb{Q}$ and $E^{(d\cdot d_K)}/\mathbb{Q}$.

\section{Examples}\label{sec:examples}

In this section we illustrate our application to Goldfeld's conjecture and the 2-part of the BSD conjecture by providing examples of $E/\mathbb{Q}$ and $K$ which satisfy Assumption (\ref{eq:star}). 

Let us first consider curves $E/\mathbb{Q}$ of rank one.

\begin{example}
  Consider the curve $37a1$ in Cremona's table, $$E=37a1: y^2 + y = x^3 - x,$$ It is the rank one optimal curve over $\mathbb{Q}$ of smallest conductor ($N=37$).  Take $$K=\mathbb{Q}(\sqrt{-7}),$$ the imaginary quadratic field with smallest $|d_K|$ satisfying the Heegner hypothesis for $N$ such that 2 is split in $K$. The Heegner point $$P=(0,0)\in E(K)$$ generates $E(\mathbb{Q})=E(K)\cong \mathbb{Z}$. Since $E$ is optimal with Manin constant 1, we know that $\omega_E$ is equal to the N\'{e}ron differential. The formal logarithm associated to $\omega_E$ is $$\log_{\omega_E}(t)=t + 1/2\cdot t^4 - 2/5\cdot t^5 + 6/7\cdot t^7 - 3/2\cdot t^8 + 2/3\cdot t^9 +\cdots$$ We have $|\tilde E(\mathbb{F}_2)|=5$ and the point $5P=(1/4, -5/8)$ reduces to $\infty \in \tilde E(\mathbb{F}_2).$  Plugging in the parameter $t=-x(5P)/y(5P)=2/5$, we know that up to a 2-adic unit, $$\log_{\omega_E}P=\log_{\omega_E}5P=2 + 2^5 + 2^6 + 2^8 + 2^9 + \cdots\in 2 \mathbb{Z}_2^\times.$$ Hence $$\frac{|\tilde E(\mathbb{F}_2)|\cdot \log_{\omega_E}P}{2}\in \mathbb{Z}_2^\times$$ and (\ref{eq:star}) is satisfied. The set $\mathcal{N}$ consists of square-free products of the signed primes $$-11, 53, -71, -127, 149, 197, -211, -263, 337, -359, 373, -379, -443, -571, -599, 613,\cdots$$ For any $d\in \mathcal{N}$, we deduce:
  \begin{enumerate}
  \item The rank part of BSD conjecture is true for $E^{(d)}$ and $E^{(-7d)}$ by Theorem \ref{thm:overK}.
  \item Since $\Delta(E)>0$, we know from Corollary \ref{cor:disc} that
  \begin{equation*}
    \begin{cases}
          \rank E^{(d)}(\mathbb{Q})=1, \quad \rank E^{(-7d)}(\mathbb{Q})=0, &  d>0, \\
          \rank E^{(d)}(\mathbb{Q})=0, \quad \rank E^{(-7d)}(\mathbb{Q})=1, & d<0.
        \end{cases}
      \end{equation*}
  \item Since $\Gal(\mathbb{Q}(E[2])/\mathbb{Q}))\cong S_3$, it follows from Theorem \ref{thm:goldfeld} that $$N_r(E, X)\gg \frac{X}{\log^{5/6}X},\quad r=0,1.$$
  \item Since  BSD(2) is true for $E/\mathbb{Q}$ and $E^{(-7)}/\mathbb{Q}$ by numerical verification, it follows from Theorem~\ref{thm:2partBSD} that the BSD(2) is true for $E^{(d)}$ and $E^{(-7d)}$ when $d>0$.
  \end{enumerate}
\end{example}

\begin{example}\label{exa:rankone}
  As discussed in \S \ref{sec:proof-4}, a necessary condition for (\ref{eq:star}) is that the local Tamagawa numbers $c_p(E)$ are all odd for $p\ne2$. Another necessary condition is that the formal group of $E$ at 2 cannot be isomorphic to $\mathbb{G}_m$: this due to the usual subtlety that the logarithm on $\mathbb{G}_m$ sends $1+2 \mathbb{Z}_2$ into $4 \mathbb{Z}_2$ (rather than $2 \mathbb{Z}_2$).  We search for rank one optimal elliptic curves with $E(\mathbb{Q})[2]=0$ satisfying these two necessary conditions. There are 38 such curves of conductor $\le300$. For each curve, we choose $K$ with smallest $|d_K|$ satisfying the Heegner hypothesis for $N$ and such that 2 is split in $K$. Then 31 out of 38 curves satisfy (\ref{eq:star}). See Table~\ref{tab:1}. The first three columns list $E$, $d_K$ and the local Tamagawa number $c_2(E)$ at 2 respectively. A check-mark in the last column means that (\ref{eq:star}) holds, in which case Theorems \ref{thm:overK}, \ref{thm:goldfeld} apply and the improved bound towards Goldfeld's conjecture holds. If $c_2(E)$ is further odd (true for 23 out of 31), then the application to BSD(2) (Theorem \ref{thm:2partBSD}) also applies.
  
  \begin{table}[h]
    \centering\caption{Assumption (\ref{eq:star}) for rank one curves}\label{tab:1}
    \begin{tabular}[h]{*{4}{|>{$}c<{$}}|}
      E & d_K & c_2(E) & \bigstar\\\hline
37a1 & -7 & 1 & \checkmark\\
43a1 & -7 & 1 & \checkmark\\
88a1 & -7 & 4 & \checkmark\\
91a1 & -55 & 1 & \checkmark\\
91b1 & -55 & 1 & \checkmark\\
92b1 & -7 & 3 & \checkmark\\
101a1 & -23 & 1 & \checkmark\\
123a1 & -23 & 1 & \checkmark\\
123b1 & -23 & 1 & \checkmark\\
124a1 & -15 & 3 & \checkmark\\
131a1 & -23 & 1 & \checkmark\\
141a1 & -23 & 1 & \checkmark\\
141d1 & -23 & 1 & \checkmark\\
\end{tabular}
\begin{tabular}[h]{*{4}{|>{$}c<{$}}|}
  E & d_K & c_2(E) & \bigstar\\\hline
148a1 & -7 & 3 & \checkmark\\
152a1 & -15 & 4 & \checkmark\\
155a1 & -79 & 1 & \checkmark\\
155c1 & -79 & 1 & \checkmark\\
163a1 & -7 & 1 & \checkmark\\
172a1 & -7 & 3 & \checkmark\\
176c1 & -7 & 2 & \checkmark\\
184a1 & -7 & 2 & \checkmark\\
184b1 & -7 & 2 & \checkmark\\
189a1 & -47 & 1 & \checkmark\\
189b1 & -47 & 1 & \checkmark\\
196a1 & -31 & 3 & \checkmark\\
197a1 & -7 & 1 & \\
\end{tabular}
\begin{tabular}[h]{*{4}{|>{$}c<{$}}|}
  E & d_K & c_2(E) & \bigstar\\\hline
208a1 & -23 & 4 & \\
208b1 & -23 & 4 & \\
212a1 & -7 & 3 & \\
216a1 & -23 & 4 & \checkmark\\
219a1 & -23 & 1 & \checkmark\\
219b1 & -23 & 1 & \checkmark\\
232a1 & -7 & 2 & \\
236a1 & -23 & 3 & \\
243a1 & -23 & 1 & \checkmark\\
244a1 & -15 & 3 & \\
248a1 & -15 & 2 & \checkmark\\
248c1 & -15 & 2 & \checkmark\\
& & & 
\end{tabular}      
  \end{table}
\end{example}

\begin{remark}
There is one CM elliptic curve in Table \ref{tab:1}: namely $E=243a1$ with $j$-invariant 0, which seems to be only $j$-invariant of CM elliptic curves over $\mathbb{Q}$ for which (\ref{eq:star}) holds.
\end{remark}

Next let us consider curves $E/\mathbb{Q}$ of rank zero.

\begin{example}
  Consider $$E=X_0(11)=11a1: y^2 +y=x^3 -x^2 -10x-20,$$ the optimal elliptic curve over $\mathbb{Q}$ of smallest conductor ($N=11$). Take $$K=\mathbb{Q}(\sqrt{-7}),$$ the imaginary quadratic field with smallest $|d_K|$ satisfying the Heegner hypothesis for $N$ such that 2 is split in $K$. The Heegner point $$P=\left(-\frac{1}{2}\sqrt{-7}+\frac{1}{2},-2\sqrt{-7}-2\right)\in E(K)$$ generates the free part of $E(K)$. Since $E$ is optimal with Manin constant 1, we know that $\omega_E$ is equal to the N\'{e}ron differential. The formal logarithm associated to $\omega_E$ is $$\log_{\omega_E}(t)=t - 1/3\cdot t^3 + 1/2\cdot t^4 - 19/5\cdot t^5 - t^6 + 5/7\cdot t^7 - 27/2\cdot t^8 + 691/9\cdot t^9 + \cdots$$ We have $|\tilde E(\mathbb{F}_2)|=5$ and the point $5P=(-\frac{3}{4}, -\frac{11}{8}\sqrt{-7}-\frac{1}{2})$ reduces to $\infty \in \tilde E(\mathbb{F}_2).$  The prime 2 splits in $K$ as $$(2)= \left(-\frac{1}{2}\sqrt{-7}+\frac{1}{2}\right)\cdot \left(\frac{1}{2}\sqrt{-7}+\frac{1}{2}\right)$$ and the parameter $t=-x(5P)/y(5P)$ has valuation 1 for both primes above 2. Plugging in $t$, we find that $$\log_{\omega_E}P\in 2\mathcal{O}_{K_2}^\times.$$ Hence $$\frac{|\tilde E(\mathbb{F}_2)|\cdot \log_{\omega_E}P}{2}\in\mathcal{O}_{K_2}^\times$$ and (\ref{eq:star}) is satisfied. The set $\mathcal{N}$ consists of square-free products of the signed primes $$-23, 37, -67, -71, 113, 137, -179, -191, 317, -331, -379, 389, -443, 449, -463, -487, -631,\cdots$$ For any $d\in \mathcal{N}$, we deduce:
  \begin{enumerate}
  \item  The rank part of BSD conjecture is true for $E^{(d)}$ and $E^{(-7d)}$ by  Theorem \ref{thm:overK}.
  \item Since $\Delta(E)<0$, we know from Corollary \ref{cor:disc} that $$\rank E^{(d)}(\mathbb{Q})=0, \quad \rank E^{(-7d)}(\mathbb{Q})=1.$$
  \item Since $\Gal(\mathbb{Q}(E[2])/\mathbb{Q}))\cong S_3$, it follows from Theorem \ref{thm:goldfeld} that $$N_r(E, X)\gg \frac{X}{\log^{5/6}X},\quad r=0,1.$$
  \item   Since  BSD(2) is true for $E/\mathbb{Q}$ and $E^{(-7)}/\mathbb{Q}$ by numerical verification, it follows from Theorem~\ref{thm:2partBSD} that the BSD(2) is true for $E^{(d)}$ and $E^{(-7d)}$.
  \end{enumerate}
\end{example}

\begin{example}
  For rank zero curves, the computation of Heegner points is most feasible when $|d_K|$ is small. Thus we fix $d_K=-7$ and search for rank zero optimal curves with $E(\mathbb{Q})[2]=0$ satisfying the two necessary conditions in Example \ref{exa:rankone} and such that $K=\mathbb{Q}(\sqrt{-7})$ satisfies the Heegner hypothesis. There are 39 such curves of conductor $\le750$. See Table~\ref{tab:2}. Then 28 out of 39 curves satisfy (\ref{eq:star}), in which case Theorems \ref{thm:overK}, \ref{thm:goldfeld} apply and the improved bound towards Goldfeld's conjecture holds. If $c_2(E)$ is further odd (true for 24 out of 28), then the application to BSD(2) (Theorem \ref{thm:2partBSD}) also applies.

  \begin{table}[h]
    \centering    \caption{Assumption (\ref{eq:star}) for rank zero curves}    \label{tab:2}
    \begin{tabular}[h]{*{4}{|>{$}c<{$}}|}
  E & d_K & c_2(E) & \bigstar\\\hline      
11a1 & -7 & 1 & \checkmark\\
37b1 & -7 & 1 & \checkmark\\
44a1 & -7 & 3 & \checkmark\\
67a1 & -7 & 1 & \checkmark\\
92a1 & -7 & 3 & \checkmark\\
116a1 & -7 & 3 & \\
116b1 & -7 & 3 & \\
176a1 & -7 & 1 & \checkmark\\
176b1 & -7 & 1 & \checkmark\\
179a1 & -7 & 1 & \checkmark\\
184d1 & -7 & 2 & \checkmark\\
232b1 & -7 & 2 & \\
268a1 & -7 & 1 & \checkmark\\
\end{tabular}
\begin{tabular}[h]{*{4}{|>{$}c<{$}}|}
  E & d_K & c_2(E) & \bigstar\\\hline  
316a1 & -7 & 1 & \\
352a1 & -7 & 2 & \checkmark\\
352e1 & -7 & 2 & \checkmark\\
368c1 & -7 & 1 & \checkmark\\
368f1 & -7 & 1 & \checkmark\\
428a1 & -7 & 3 & \\
464c1 & -7 & 2 & \\
464d1 & -7 & 1 & \\
464f1 & -7 & 1 & \\
464g1 & -7 & 2 & \\
557b1 & -7 & 1 & \checkmark\\
568a1 & -7 & 1 & \\
571a1 & -7 & 1 & \\
\end{tabular}
\begin{tabular}[h]{*{4}{|>{$}c<{$}}|}
  E & d_K & c_2(E) & \bigstar\\\hline  
592b1 & -7 & 1 & \checkmark\\
592c1 & -7 & 1 & \checkmark\\
659b1 & -7 & 1 & \checkmark\\
688b1 & -7 & 2 & \checkmark\\
701a1 & -7 & 1 & \checkmark\\
704c1 & -7 & 1 & \checkmark\\
704d1 & -7 & 1 & \checkmark\\
704e1 & -7 & 1 & \checkmark\\
704f1 & -7 & 1 & \checkmark\\
704g1 & -7 & 1 & \checkmark\\
704h1 & -7 & 1 & \checkmark\\
704i1 & -7 & 1 & \checkmark\\
739a1 & -7 & 1 & \checkmark\\
    \end{tabular}
  \end{table}
\end{example}

\begin{remark}\label{rem:assumptionstar}
  Even when $E$ does not satisfy  (\ref{eq:star}) for any $K$ (e.g., when $E(\mathbb{Q})$ has rank $\ge2$ or $\Sha(E/\mathbb{Q})[2]$ is nontrivial), one can still prove the same bound in Theorem \ref{thm:goldfeld} by exhibiting \emph{one} quadratic twist $E^*$ of $E$ such that $E^*$ satisfies (\ref{eq:star}) (as quadratic twisting can \emph{lower} the 2-Selmer rank). We expect that one can always find such $E^*$ when the two necessary conditions ($c_p(E)$'s are odd for $p\ne2$ and $a_2(E)$ is even) are satisfied, and so we expect that Theorem \ref{thm:goldfeld} applies to a large positive proportion of elliptic curves $E$. Showing the existence of such $E^*$ amounts to showing that the value of the anticyclotomic $p$-adic $L$-function at the trivial character is nonvanishing mod $p$ among quadratic twists families for $p=2$. This nonvanishing mod $p$ result seems to be more difficult and we do not address it here (but when $p\ge5$ see Prasanna \cite{Prasanna2010} and the forthcoming work of Burungale--Hida--Tian).
\end{remark}

\bibliographystyle{alpha}
\bibliography{Congruence}

\end{document}